\documentclass[11pt, letter]{article}       
\usepackage{graphicx}
\usepackage{mathptmx}      
\usepackage{amsmath,amssymb,xcolor, enumerate,amsthm}
\usepackage{bm}

\usepackage[margin=1in]{geometry}
\usepackage[utf8]{inputenc}
\usepackage{hyperref}
\usepackage[linesnumbered,algoruled,boxed,lined]{algorithm2e}
\usepackage{subcaption}
\usepackage{standalone}
\usepackage{authblk}
\usepackage{tikz}
\usepackage{enumitem}
\usepackage{hyperref}
\usepackage{mathtools}
\usepackage{subcaption}

\usepackage[backend=bibtex]{biblatex}
\bibliography{bib}
\usepackage{nicefrac}
\usepackage{algpseudocode}
\usetikzlibrary{arrows, calc, shapes}
\usetikzlibrary{decorations.pathreplacing,calligraphy}

\mathchardef\mhyphen="2D

\newcommand{\R}{\mathbb{R}}
\newcommand{\Z}{\mathbb{Z}}
\newcommand{\Q}{\mathbb{Q}}

\newcommand{\dynN}{\mathcal{N}}

\newcommand{\N}{\mathcal{N}}
\newcommand{\head}{\operatorname{head}}
\newcommand{\tail}{\operatorname{tail}}

\newcommand{\rep}{\operatorname{rep}}

\newcommand{\TENS}[2]{\dynN^{[#1,#2]}}
\newcommand{\TEN}[1]{\TENS{1}{#1}}
\DeclareMathOperator*{\argmin}{arg\,min}
\newcommand{\lift}{\operatorname{lift}}
\newcommand{\low}{\operatorname{low}}
\newcommand{\high}{\operatorname{high}}
\DeclareMathOperator*{\argmax}{arg\,max}
\newcommand{\start}{\operatorname{start}}
\newcommand{\cut}{\operatorname{cut}}
\renewcommand{\phi}{\varphi}

\makeatletter
\DeclareRobustCommand{\cev}[1]{%
  {\mathpalette\do@cev{#1}}%
}
\newcommand{\do@cev}[2]{%
  \vbox{\offinterlineskip
    \sbox\z@{$\m@th#1 x$}%
    \ialign{##\cr
      \hidewidth\reflectbox{$\m@th#1\vec{}\mkern4mu$}\hidewidth\cr
      \noalign{\kern-\ht\z@}
      $\m@th#1#2$\cr
    }%
  }%
}
\makeatother

\usepackage[linecolor=red,backgroundcolor=red!25,bordercolor=red]{todonotes}

\newtheorem{theorem}{Theorem}
\newtheorem{definition}{Definition}
\newtheorem{lemma}[theorem]{Lemma}

\newtheorem{obs}{Observation}
\theoremstyle{definition}

\newtheorem{rem}{Remark}

\begin{document}

\title{Minimum-Cost Flows Over Time}
\date{}

\author[1]{Miriam Schl\"oter}
\author[1]{Robert Weismantel}

\affil[1]{Department of Mathematics, Institute for Operations Research, ETH Z\"urich, Z\"urich, Switzerland}
\maketitle
\begin{abstract}
In this paper we show that every maximum minimum-cost flow over time problem has an optimal solution with a repeated structure if the given time horizon is large enough.
\end{abstract}
\thispagestyle{empty}
\setcounter{page}{1}
\section{Introduction}

Time is a critical resource in many network routing problems arising in road, pedestrian, rail, or air
traffic control, including evacuation planning as one important example~\cite{HamacherTjandra02a}. Network
flows over time capture the essence of these applications as they model the variation of flow along arcs
over time as well as the delay experienced by flow traveling at a given pace through the network. Other
application areas include, for instance, production systems, communication networks, and financial
flows~\cite{Aronson89,PowellJO95}. 

\paragraph{Maximum Flows Over Time.}
In the following $\Z_+, \Q_+, \R_+$ denote all $x$ in $\Z, \Q$ or  $\R$ with $x \geq 0$.
The study of flows over time goes back to the work of Ford and Fulkerson~\cite{Ford1958,Ford1962}. 
A \emph{flow over time network} $\dynN = (D,s,t,u,\tau)$ consists of a directed graph $D = (V,A)$ with integral \emph{arc capacities}~$u\in \Z_+^A$, integral \emph{arc transit times}~$\tau\in \Z_+^A$, a \emph{source node} $s \in V$, and a \emph{sink node} $t \in V$.
The transit time of an arc specifies how long flow needs to travel from the tail of an arc to its head.
A flow over time $f$ in $\dynN$ is a function $f:A\times \Z_{+} \rightarrow \Z_+$ that specifies a flow value $f(a,\theta)$ for each arc $a \in A$ and every point in time $\theta \in \{1,2,\ldots,\infty\}$.
Throughout this paper we consider flows over time in this \emph{discrete} time model. Another common model is the \emph{continuous} time model where a flow over time is a Lebesque integrable function $f:A \times [0,\infty) \rightarrow \R_+$.
However, our main results also holds for the continuous time model.
A flow over time respects the capacity on each arc at every point in time, and flow conservation on the \emph{intermediate nodes} $V \setminus \{s,t\}$.
That is, flow that enters some intermediate node $v$ has to leave it eventually after potentially waiting for some time at $v$.
We say that a flow over time $f$ in $\dynN$ has \emph{time horizon} $\theta \in \Z_+$ if no arc carries flow after time $\theta$.
For a more in-depth introduction to flows over time see~\cite{Skutella-Intro09}.
Given a flow over time network $\dynN$, Ford and Fulkerson consider the \emph{Maximum Flow Over Time
Problem}, that is, to send the maximum possible amount of flow from~$s$ to~$t$ within the given time horizon~$\theta\in\Z_{+}$. 
Ford and Fulkerson observe that this problem can be reduced to a static maximum flow problem in an
exponentially large \emph{time-expanded network} $\dynN^{[1,\theta]}$, whose node set consists of~$\theta$ copies of the given node set~$V$ that we call \emph{layers}.
In particular, they show that a static flow $x$ in the time-expanded network $\TEN{\theta}$ corresponds to a flow over time $f$ with time horizon $\theta$ and vice-versa.
More importantly, Ford and Fulkerson~\cite{Ford1958,Ford1962} show that an $s$-$t$ path decomposition of static min-cost $s$-$t$-flow with arc transit times as costs in the given network $\dynN$ yields a maximum flow over time by
repeatedly sending flow along these $s$-$t$-paths at the corresponding flow rates. 
The resulting maximum $s$-$t$ flow over time has the property that the flow value on each arc is constant almost always, provided that the time horizon is large enough.
We call a flow over time with this property \emph{repeated}.

%
%
%
\paragraph{Minimum-Cost Flows Over Time.} Given a flow over time network $\dynN$, costs $c \in \Z^{A}$ on the arcs, a value $v \in \Q_+$, and a time-horizon $\theta \in \Z_+$, the corresponding \emph{minimum-cost $s$-$t$ flow over time with value $v$}, is a flow over time with value $v$ and time-horizon $\theta$ of minimal cost. 
In contrast to the classical minimum-cost flow problem, the minimum-cost flow over time problem is already NP-hard~\cite{Klinz95}.
Fleischer and Skutella~\cite{Fleischer2007} show that there always exists a minimum-cost flow over time where flow never waits at an intermediate node, even for a setting with multiple sources and sinks with given supplies and demands, respectively.
Fleischer and Skutella~\cite{Fleischer2007} also introduce a condensed variant of the time-expanded network of polynomial size that leads to a fully polynomial-time approximation scheme for the minimum-cost flow over time problem.
For the special case where the arc costs are proportional to the transit times, Fleischer and Skutella~\cite{Fleischer2007} describe a simple fully polynomial-time approximation scheme for the minimum-cost $s$-$t$ flow over time problem  based on capacity scaling that does not rely on any form of time-expansion.
\paragraph{Our Contribution.}
One structural question about minimum-cost flow over time problems that so far was unanswered is whether every minimum-cost flow over time problem has an optimal solution that is \emph{repeated}.
The main result of this paper is a positive answer to this question for the maximum minimum-cost flow over time problem.
\begin{theorem}\label{thm:main}
	Every maximum minimum-cost flow over time problem corresponding to a flow over time network $\dynN$, costs $c \in \Z^{A}$, and a time-horizon $\theta \in \Z_+$ has an optimal solution that is repeated provided that $\theta > 2(\theta^1_{j_2} + J(\theta^1_{j_2}) + J(\theta^1_{j_2} + J(\theta^1_{j_2}))$ with $\theta^1_{j_2}$ defined as in~\eqref{eq:height_non_repeated} and $J(\cdot)$ defined as in~\eqref{eq:def_bound}.
\end{theorem}
Throughout the whole paper we assume that the given time-horizon $\theta$ fulfills
\begin{align}\label{eq:bound_on_TH}
    \theta > 2(\theta^1_{j_2} + J(\theta^1_{j_2}) + J(\theta^1_{j_2} + J(\theta^1_{j_2})).
\end{align}
%
%

To prove Theorem~\ref{thm:main} it suffices to consider the case where $v$ is the value of a maximum flow over time with time-horizon $\theta$. 
The proof for smaller values of $v$ follows directly from this special case.
Our proof of Theorem~\ref{thm:main} is constructive.
Our of departure is the existence of a repeated maximum flow over time $f_0$ with time horizon $\theta$ computed by the algorithm of Ford and Fulkerson~\cite{Ford1962}.

\emph{Using the correspondence shown by Ford and Fulkerson~\cite{Ford1962} we in the following interpret every flow over time as a flow in the time-expanded network $\TEN{\theta}$. We further just consider networks with $\tau \in \{0,1\}^A$. This is without loss of generality by subdividing arcs $a$ with $\tau(a)>1$ into multiple arcs unit transit time by inserting additional nodes.}

To eventually obtain a maximum flow over time with time horizon $\theta$ of minimum cost, we iteratively augment along cycles with negative cost in the current residual network of $\TEN{\theta}_f$ in such a way that after every augmentation the resulting flow over time is still repeated.
Let $f$ be a repeated maximum $s$-$t$ flow over time with time horizon $\theta$. 
The following two main ingredients are central throughout our construction:
\begin{enumerate}
    \item The residual network $\TEN{\theta}_f$ corresponding to $f$ has a repeated structure: If $f$ is repeated during the interval $[\theta_1,\theta_2]$, then there is the same set of arcs between layers $t$ and $t+1$ for all $t \in \{\theta_1,\ldots,\theta_2-1\}$ (see Observation~\ref{obs:layers_are_the_same}).
    Additionally $f$ induces a (static) maximum $s$-$t$ flow $\phi(f)$ in $\dynN$ (see~\eqref{eq:def_of_flow_projection} for the definition of $\phi(f)$, and Lemma~\ref{lem:phif_is_feasible} and Lemma~\ref{lem:maxflow_induces_max_flow}).
    \item There is a correspondence between negative cost cycles in the repeated part of $\TEN{\theta}_f$ and certain negative cost cycles in $\dynN_{\phi(f)}$ via a Projection and a Lifting Lemma (see Lemma~\ref{lem:projection_lemma} and Lemma~\ref{lem:lifing_lemma}).
    %
    %
\end{enumerate}
Overall, the correspondence induced by the projection and the lifting lemma, together with the fact that $\TEN{\theta}_{f_0}$ has a repeated structure allows us to iteratively choose specific negative cost cycles in the current residual network $\dynN_{\phi(f)}$ and to lift them to collections of negative cost cycles in the repeated part of $\TEN{\theta}_f$ such that augmenting along these collections results in a repeated flow over time.
This first part of our construction results in a flow over time $f$ with time-horizon $\theta$ with the property that the repeated part of $\TEN{\theta}_f$ does not contain any cycle of negative cost and such that $\dynN_{\phi(f)}$ has some useful structural properties, see Lemma~\ref{Cor:no_neg_transit_time_cycle} and Lemma~\ref{lem:properties}
 These properties allow us to analyze the second part of our construction where we choose a specific order of augmentation steps along the remaining cycles of negative cost in $\TEN{\theta}_f$.
Such cycles $C$ can be of different forms.
In the second step we only augment along negative cost cycles $C$ in $\TEN{\theta}_f$ that do not intersect with the repeated part of $\TEN{\theta}_f$ and \emph{both} non-repeated parts of $\TEN{\theta}_f$.
See Figure~\ref{fig:cycles1} for examples of such cycles.
\begin{figure}
    \centering
    \begin{subfigure}{0.4\textwidth}
    \begin{center}
            \includegraphics{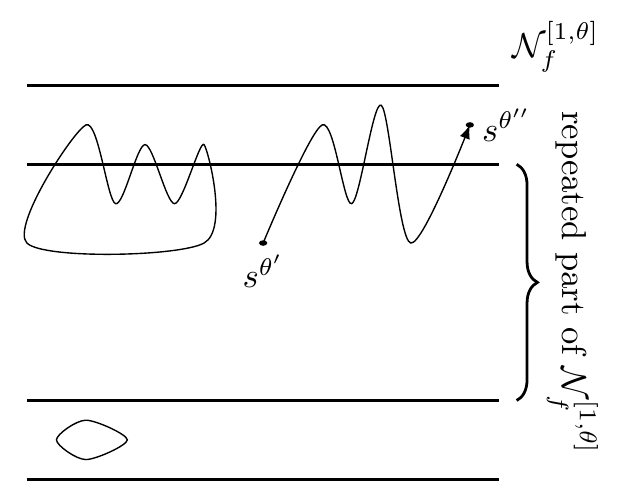}
    \end{center}
    \caption{Cycles along which we augment in the second step of our construction}
\label{fig:cycles1}
    \end{subfigure} 
    \hfill
\begin{subfigure}{0.5\textwidth}
    \begin{center}
            \includegraphics{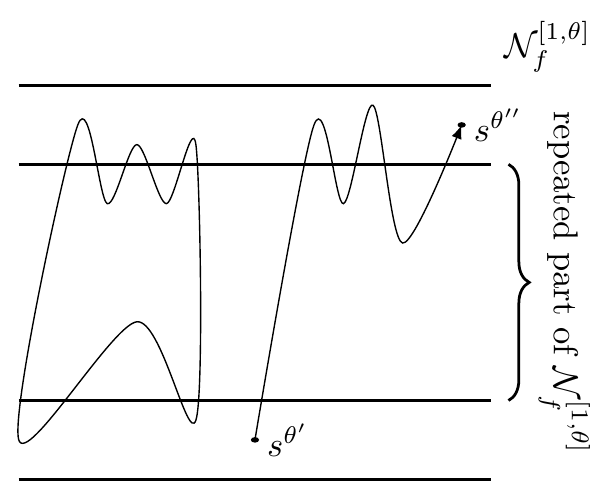}
    \end{center}
    \caption{Cycles crossing all three ``areas'' of $\TEN{\theta}_f$\\ \hfill}
\label{fig:cycles2}
\end{subfigure}
\caption{Examples of cycles in $\TEN{\theta}_f$. The paths shown in both figures become cycles by connecting them to the super source of the network. We omitted this connection for the sake of visibility.}
\end{figure}
We prove that such a cycle $C$ can only intersect with a bounded number of layers of the repeated part of $\TEN{\theta}_f$.
To deduce this fact, we exploit that the projection of the part of $C$ that lies in the repeated part of $\TEN{\theta}_f$ has a special structure due to the structural properties of $\dynN_{\phi(f)}$.  
For the other possible cases depicted in Figure~\ref{fig:cycles2} we develop similar arguments to deduce that they cannot occur.
%
%
%
By special additional augmentations that we perform in the second part of our construction we make sure that the cycles along which we augment cannot ``stack on top of each other'' to eventually result in a non-repeated flow over time.
Overall, this implies that our construction terminates with a maximum flow over time of minimum cost that is repeated if the time-horizon $\theta$ is large enough.

\section{Preliminaries}
%
Let $D = (V,A)$ be a directed graph. 
A sequence $W = (v_1,a_1,v_2,a_2,\ldots,v_{k-1},a_{k-1},v_k)$ with $a_i = (v_i,v_{i+1}) \in A$ for all $i \in \{1,\ldots,k-1\}$ and $v_i \in V$ for all $i \in \{1,\ldots,k\}$ is called a $\emph{walk}$ in $D$.
The walk $W$ is \emph{closed} if $v_1 = v_k$.
A walk $W$ is called a \emph{path} if we additionally have $v_i \neq v_j$ for all $1\leq i<j\leq k$.
In this case, $W$ is also called an $v_1$-$v_k$ path.
If $W$ is a path and $v_1 = v_k$, then $W$ is called a \emph{cycle}.
The proof of the following lemma is straight forward.
\begin{lemma}\label{lem:decomposition_of_walk}
    Let $D = (V,A)$ be a directed graph and let $W = (v_1,a_1,\ldots,a_{k-1},v_k)$ be a walk in $D$. 
    If $W$ is a closed, then $W$ can be decomposed into disjoints sets $B_1,\ldots, B_\ell$ such that $B_i$ is a cycle for all $i \in \{1,\ldots,\ell\}$.
    If $W$ is a non-closed walk, then $W$ can be decomposed into a $v_1$-$v_k$ path $P$ and cycles $B_1,\ldots, B_\ell$.
\end{lemma}
Let $W_1 = (v_1,a_1,\ldots,a_{k-1},v_k)$ and $W_2 = (u_1,b_1,\ldots,b_{\ell-1},u_\ell)$ be two walks in $D$ with $v_k = u_1$. 
By $W_1 \oplus W_2$ we denote the concatenation of both walks.

Let $P = (v_1,a_1,v_2,a_2,\ldots,v_{k-1},a_{k-1},v_k)$ be a path in $D$.
The \emph{subpath} $P[v_i,v_{i'}]$ with $1 \leq i < i' \leq v_k$ is defined by $P[v_i,v_{i'}] = (v_i,a_i,\ldots,v_{i'-1},a_{i'-1},v_{i'})$

We say that the directed graph $D$ is \emph{Eulerian} if for every vertex of $D$ in-degree and out-degree are equal.
The following fact is well-known.
\begin{lemma}\label{lem:Eulerian_Subgraph_Decomp}
    Let $D = (V,A)$ be an Eulerian graph. 
    Then $A$ can be decomposed into the arc sets of pairwise disjoint cycles in $D$.
\end{lemma}
\paragraph{The Time-Expanded Network and Flows.}
One easy approach to solve many flow over time problems is to translate them back to classical static flow problems by using the \emph{time-expanded network} $\TEN{\theta}$ corresponding to a given flow over time network $\dynN$ and a time-horizon $\theta$.
The time-expanded network $\TEN{\theta} = (D^\theta = (V^\theta, A^\theta),u^\theta, \hat{s},\hat{t})$ consists of $\theta$ copies of the node set $V$, called \emph{layers}.
We denote the copy of node $v \in V$ in layer $\theta' \in \{1,\ldots,\theta\}$ by $v^{\theta'}$.
For each arc $(u,v) = a  \in A$ the time-expanded network contains $\theta-1$ copies $a^{\theta'} = (u^{\theta'}, v^{\theta'+1})$ for all $\theta' \in \{1,\ldots,\theta-1\}$.
%
%
Additionally, the time-expanded network also contains a super-source $\hat{s}$ that is connected to $s^{\theta'}$ for all $\theta' \in \{1,\ldots,\theta\}$ and a super-sink $\hat{t}$ that is connected to $t^{\theta'}$ for all $\theta' \in \{1,\ldots,\theta\}$, i.e., $V^\theta = \{v^{\theta'}: v \in V \text{ and } \theta' \in \{1,\ldots,\theta\}\} \cup \{\hat{s},\hat{t}\}$, 
and
 \begin{align*}
 	 	A^\theta = &\{a^{\theta'}: a \in A \text{ and } \theta' \in \{1,\ldots,\theta-1\}\} \cup\{(\hat{s},s^{\theta'}):\theta' \in \{1,\ldots,\theta\}\} \cup \{(t^{\theta'},\hat{t}):\theta' \in \{1,\ldots,\theta\} \}.
\end{align*}
The capacity function $u^\theta$ is defined as follows: All the arcs leaving $\hat{s}$ or entering $\hat{t}$ have infinite capacity, while each copy $a^{\theta'}$ with $\theta' \in \{1,\ldots,\theta-1\}$ of an arc $a \in A$ gets the same capacity as $a$, i.e. $u^\theta(a^{\theta'}) := u(a)$ for all $\theta' \in \{0,\ldots,\theta-1\}$.
If we are additionally given costs $c \in \Z^A$, then the corresponding cost $c^\theta$ function in the time-expanded network is defined by given all the arcs leaving $\hat{s}$ or entering $\hat{t}$ have zero cost, while each copy $a^{\theta'}$ with $\theta' \in \{1,\ldots,\theta-1\}$ of an arc $a \in A$ gets the same cost as $a$, i.e. $c^\theta(a^{\theta'}) := c(a)$ for all $\theta' \in \{0,\ldots,\theta-1\}$.
A flow over time $f$ with time-horizon $\theta \in \Z_+$ in a dynamic network $\dynN$ naturally induces a (static) flow in the corresponding time-expanded network $\TEN{\theta}$ and vice-versa.
\emph{Recall that, throughout this paper, whenever we speak of a flow over time $f$ in $\dynN$ with time horizon $\theta$, $f$ refers to the corresponding (static) flow in the time-expanded network.
}
Instead of ``$f$ is a flow over time in $\TEN{\theta}$ (or $\dynN$) with time horizon $\theta$'' we also just say ``$f$ is a flow over time in $\TEN{\theta}$''.
\emph{Also, whenever it is clear from context that we consider a time-expanded network, we write $u$ and $c$ instead of $u^\theta$ and $c^\theta$.}
%

Note that the time-expanded network as defined above only models flows over time where waiting at intermediate nodes is not allowed. 
To model waiting special holdover arcs connected adjacent copies of each node can be introduced.
For our purpose, constructing minimum-cost flows over time, it suffices to consider flows over time without waiting as Fleischer and Skutella showed that there always exists a minimum-cost flow over time without waiting~\cite{Fleischer2007}.

For $\theta_a, \theta_b \in \{1,\ldots,\theta\}$ with $\theta_a \leq \theta_b$ we denote by $\TENS{\theta_a}{\theta_b}$ the subnetwork of $\TEN{\theta}$ induced by the following set of nodes
\begin{align*}
    \{v^{\theta'}:v \in V \text{ and } \theta' \in \{\theta_a,\ldots, \theta_b\}\}\cup\{\hat{s},\hat{t}\},
\end{align*}
i.e., the arcs of $\TENS{\theta_a}{\theta_b}$ are the arcs of $\TEN{\theta}$ that are connected to $\hat{s}$ or $\hat{t}$ or lie between layers $\theta_a$ and $\theta_b$ of $\TEN{\theta}$.
We say that a flow over time $f$ is \emph{repeated during $[\theta_1,\theta_2]$}  with $\theta_1,\theta_2\in \Z_+$, $\theta_1 < \theta_2\leq \theta$, and $\theta_2-\theta_1 \geq 2$ if
$f(a^{\theta'}) = f(a^{\theta'+1})$ for all $a \in A$ and $\theta' \in \{\theta_1,\ldots,\theta_2-1\}$.

Let $\dynN = (D = (V,A),s,t,u,\tau)$ be a flow over time network.
Throughout the paper we often consider \emph{static} flows in $\dynN$ by just ignoring the transit times.

Let $\mathcal{D} = (D=(V,A),u, s,t)$ be some network (for example a flow over time network or a time-expanded network), with capacities $u$, a source $s$, a sink $t$, costs $c$ and potentially transit times $\tau$ on the arcs.
Given an arc $a = (u,v) \in A$, $\cev{a} = (v,u)$ denotes the corresponding \emph{backward arc}.
    We also define $\cev{A} = \{\cev{a}:a \in A\}$.    
    Given a flow $x$ in $\mathcal{D}$, we say that an arc $a \in A$ is \emph{fully congested} by $f$ if $f(a) = u(a)$.
     We define the \emph{residual network} corresponding to $x$ by $\mathcal{D}_x = (D_x = (V, A_x),s,t,u_x)$ as follows,
    \begin{align*}
        u_x(a):= u(a)-x(a) \text{ for all } a \in A \text{ and } u_x(\cev{a}) := x(a) \text{ for all } \cev{a} \in \cev{A} \text{ and }
        A_x := \{a \in A \cup \cev{A}:u_x(a)>0\}.
    \end{align*}
    The cost function $c$ is extended to $\mathcal{D}_x$ by defining $c_x(a) := c(a)$ for all $a \in A$ and $c_x(\cev{a}) = -c(a)$ for all $a \in \cev{A}$.
    The potential transit times are extended to the residual network in the same way. 
    \emph{Again, whenever the considered network is clear from contex, we write $u$, $c$ and $\tau$ instead of $u_x$, $c_x$ and $\tau_x$.}

    Let $S$ be a subnetwork of $\mathcal{D}_x$. 
    We denote the arcs of the subnetwork by $A(S)$ and the nodes in the subnetwork by $V(S)$.
    Further, the \emph{incidence vector} $\chi^S \in \{0,1\}^{A_x}$ of $S$ is given by
    \begin{align*}
        \chi^S(a) := 
        \begin{cases}
            1 \text{ if } a \in A(S)\\
            0 \text{ if } a \not \in A(S).
        \end{cases}
    \end{align*}
    Let $B$ be a cycle in $\mathcal{D}_x$. 
    We define the \emph{capacity} $u(B)$ of $B$ by $u(B) := \min\{u_x(a):a \in A(B) \}$.
    The capacity of a path $P$ in $\mathcal{D}_x$ is defined analogously.
    \emph{Augmenting $x$ along a cycle $B$} means adding $u_x(B) \cdot \chi^{B}$ to the flow $x$, i.e, the resulting static flow is $x + u_x(B) \cdot \chi^{B}$.

    Given $\mathcal{D}_x$, we denote by $\mathcal{D}_x^{u = 1}$ the network obtained from $\mathcal{D}_x$ as follows:
    For every arc $a \in A(\mathcal{D}_x)$, introduce $u(a)$ copies of $a$ with capacity $1$.
    Let $H$ be an Eulerian subgraph of $\dynN_x^{u = 1}$.
    Let $B_1,\ldots,B_k$ be a decomposition of $A(H)$ into pairwise arc-disjoint cycles, according to Lemma~\ref{lem:Eulerian_Subgraph_Decomp}. 
    By the expression \emph{augmenting $x$ along $H$}, we mean updating $x$ to
    $x + \sum_{i = 1}^k\chi^{B_i}$.




%
%
A central ingredient of our proof of Theorem~\ref{thm:main} is the fact that for a repeated flow over time $f$ with time horizon $\theta$ the corresponding residual network also has a repeated structure.
\begin{obs}\label{obs:layers_are_the_same}
    If $f$ is a flow over time in $\TEN{\theta}$ that is repeated during $[\theta_1,\theta_2]$, then the networks $\TENS{\theta_a}{\theta_a+1}_f$ and $\TENS{\theta_b}{\theta_b+1}_f$ are isomorphic for all $\theta_a,\theta_b \in \{\theta_1,\ldots,\theta_2-1\}$.    
\end{obs}
Let $f$ be a flow over time in $\TEN{\theta}$ that is repeated during $[\theta_1,\theta_2]$.
For a walk $W$ in $\TENS{\theta_1}{\theta_2}_f$ we define its \emph{height} $h^{\theta}_f(W)$ as the number of layers in the time-expanded network that contribute to $W$.
More formally, 
\begin{align*}
    h^\theta_f(W) &:=(h^\theta_f)^+(W) -(h^\theta_f)^-(W) \text{ with }\\
   (h^\theta_f)^+(W) &:= \max\{\theta':\theta' \in \{\theta_1,\ldots,\theta_2\} \text{ and there exists } v \in V \text{ such that } v^{\theta'} \in V(W) \} \text{ and }\\
    (h^\theta_f)^-(W) &:= \min\{\theta':\theta' \in \{\theta_1,\ldots,\theta_2\} \text{ and there exists } v \in V \text{ such that } v^{\theta'} \in V(W) \}.
\end{align*}
See Figure~\ref{Fig:Height} for an illustration of the height of a walk in $\TENS{\theta_1}{\theta}_f$.
\begin{figure}[ht]
    \begin{center}
       \includegraphics{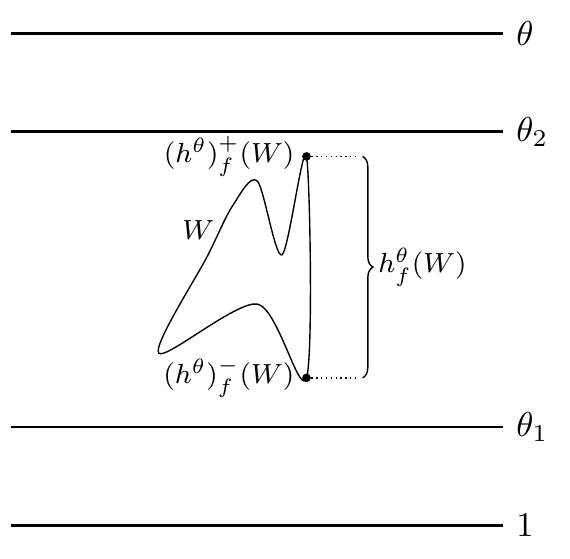}
    \end{center}
    \caption{The height of a walk}
    \label{Fig:Height}
\end{figure}

%
%
Similarly, we can also define the height of a walk $W$ in $\dynN$:
Let $x$ be an $s$-$t$ flow in $\dynN$. 
Let $W = (v_1,a_1,v_2,\ldots,a_{\ell-1},v_\ell)$ be a walk in $\dynN_x$.
For every $i \in \{1,\ldots,\ell\}$ we define the \emph{height of $v_i$ with respect to $W$ and $x$} by
\begin{align}\label{eq:height_of_node_wrt_path}
    h_x(v_i,W) = \sum_{j = 1}^{i-1} \tau(a_j),
\end{align}
and
\begin{align*}
    h_x(W) &= h_x^+(W) - h_x^-(W) \text{ with }\\
  h_x^-(W) &= \min_{i \in \{1,\ldots,\ell\}} h_x(v_i,W) \quad \text{and} \quad \low(W) = \argmin_{v \in V(W)} h_x(v,W) \\
  h_x^+(W) &= \max_{i \in \{1,\ldots,\ell\}} h_x(v_i,W) \quad \text{and} \quad \high(W) = \argmax_{v \in V(W)} h_x(v,W).
\end{align*}

Let $f$ be a flow over time that is repeated during $[\theta_1,\theta_2]$.
Via the projection $\phi$, the flow over time $f$ can be mapped to a static flow in $\dynN$ as follows,
\begin{align}\label{eq:def_of_flow_projection}
    \phi(f)(a) := f(a^{\theta_1}) \text{ for all } a \in A.
\end{align}
%

%
\begin{lemma}\label{lem:phif_is_feasible}
   Let $f$ be a flow over time in $\TEN{\theta}$ that is repeated during $[\theta_1,\theta_2]$.
   Then, $\varphi(f)$ is a feasible static flow in $\dynN$.
\end{lemma}
\begin{proof}
    First, we have $\phi(f)(a) = f(a^{\theta_1}) \leq u^\theta(a^{\theta_1}) = u(a) $ for all $a \in A$.
    Thus, $\phi(f)$ respects the capacities $u$.
    It remains to check that flow conservation holds in $\dynN$.
    Denote by $\delta^+_j(\dynN,v)$ respectively $\delta^-_j(\dynN,v)$ the set of all outgoing respectively incoming arcs of a node $v \in V$ with transit time $j$ with $j \in \{0,1\}$ in the network $\dynN$.
    Similarly $\delta^+(\TEN{\theta},v)$ and $\delta^-(\TEN{\theta},v)$ denote the sets of outgoing respectively incoming arcs of a node $v \in V^\theta$ in the time-expanded network $\TEN{\theta}$.
    \begin{align*}
        \sum_{a \in \delta^-(\dynN,v) } \phi(f)(a) &= \sum_{a \in \delta^-(\dynN,v)} f(a^{\theta_1}) 
        = \sum_{a \in \delta_0^-(\dynN,v)} f(a^{\theta_1}) + \sum_{a \in \delta_1^-(\dynN,v)} f(a^{\theta_1}) \\
        &= \sum_{a \in \delta_0^-(\dynN,v)} f(a^{\theta_1+1}) + \sum_{a \in \delta_1^-(\dynN,v)} f(a^{\theta_1})
        = \sum_{a' \in \delta^-(\TEN{\theta},v^{\theta_1+1})}f(a')\\
        &= \sum_{a' \in \delta^+(\TEN{\theta},v^{\theta_1+1})}f(a')
        = \sum_{a \in \delta^+(\dynN,v)}f(a^{\theta_1+1}) = \sum_{a \in \delta^+(\dynN,v)}f(a^{\theta_1}) = \sum_{a \in \delta^+(\dynN,v)} \phi(f)(a).
    \end{align*}
    In the third and second to last equation we use that $f$ is repeated while in the fifth equation we use that $f$ fulfills flow conservation at every non-terminal node of $\TEN{\theta}$.
\end{proof}
Our point of departure is the following important result of Ford and Fulkerson.
\begin{lemma}\label{lem:maxflow_induces_max_flow}
    If $f$ is a repeated maximum flow over time computed by the algorithm of Ford and Fulkerson~\cite{Ford1958,Ford1962} and $\theta > \sum_{a \in A}\tau(a)$, then
    \begin{itemize}
        \item $f$ is repeated during $[\theta_1,\theta_2]$ with $\theta_1 < \sum_{a \in A} \tau(a)$ and $\theta_2 > \theta- \sum_{a \in A} \tau(a) $,
        \item  $\phi(f)$ is a (static) maximum $s$-$t$ flow in $\dynN$, and
        \item there are no cycles $C$ with $\tau(C) < 0$ in $\dynN_{\phi(f)}$.
    \end{itemize}
\end{lemma}
\begin{proof}
    $f$ is computed via the algorithm of Ford and Fulkerson that works as follows:
    Create a network $\dynN'$ by connecting $t$ to $s$ with an additional arc $a'$ of infinite capacity and `transit time' $-\theta$ and compute a minimum-cost circulation $x'$ in $\dynN'$ with the transit times as cost. 
    Clearly, $x'$ induces an $s$-$t$ flow $x$ in $\dynN$.
    Denote by $\mathcal{P}$ the set of all $s$-$t$ paths in $\dynN$
    Let $(x_P)_{P \in \mathcal{P}}$ be a path decomposition of $x$.
    A maximum flow over time is obtained by sending flow into each $P \in \mathcal{P}$ at rate $u_P$ during $\{1,\ldots,\theta-\tau(P)\}.$
    Thus, by construction of the maximum flow over time, the value of $\phi(f)$ is equal to the value of $x$.
    It remains to verify that $x$ is a maximum $s$-$t$ flow. 
    Suppose $x$ is not a maximum $s$-$t$ flow. 
    Then there is an augmenting $s$-$t$ path $P$ in $\dynN_x$ with $\tau_x(P) \leq \sum_{a \in A}\tau(a)$.
    Thus, $P \cup \{a'\}$ is a cycle in $\dynN'_{x'}$ with a strictly negative transit time, contradicting the fact that $x'$ is a minimum-cost circulation.
\end{proof}
Let $f$ be a flow over time in $\TEN{\theta}$ that is repeated during $[\theta_1,\theta_2]$.
A cycle $B$ in $\dynN_{\phi(f)}$ is called \emph{$v$ - reachable} for some $v \in V$ if there exist an $v$-$u$ path and an $u$-$v$ path in $\dynN_{\phi(f)}$ for some node $u \in V(B)$.

In our proof of Theorem~\ref{thm:main} we distinguish between negative cost cycles in the current residual network $\dynN_x$ that lie ``between'' two minimum $s$-$t$ cuts and negative cost cycles that are reachable from $s$ or from which $t$ can be reached.
To make this formal, let $x$ be a static maximum $s$--$t$ flow in $\dynN$.
By $\dynN_{x}(s)$,$\dynN_{x}(t)$, and $\overline{\dynN}_{x}$ we denote the subnetworks of $\dynN_{x}$ induced by the following sets of nodes, respectively.
\begin{align*}
    V_{x}(s) &= \{v: v \in V \text{ such that there exists an $s$-$v$}  \text{ path in } \dynN_{x} \}\\
    V_{x}(t) &= \{v: v \in V \text{ such that there exists an $v$-$t$} \text{ path in } \dynN_{x} \} \\
    \overline{V}_{x} &= V \setminus \{V_x(s) \cup V_x(t)\} = \{\text{nodes that lie \emph{between} minimum $s$-$t$ cuts in $\dynN$.}\}.
\end{align*}
Since $x$ is a maximum $s$-$t$ flow in $\dynN$ and thus all minimum $s$-$t$ cuts in $\dynN$ are completely saturated by $x$, we remark:
\begin{obs}\label{obs:disjoint}
    We have $V_x(s) \cap V_x(t) = \emptyset$ and a cycle $C$ in $\dynN_x$ either lies completely in $\dynN_x(s)$, $\overline{\dynN}_x$ or $\dynN_x(t)$.
\end{obs}
The following lemma is due to the fact that augmenting along cycles preserves the value of a flow and hence does not change the reachability relations between nodes.
\begin{lemma}\label{lem:invariant_under_augmentations}
    Let $x$ be a maximum $s$--$t$ flow in $\dynN$ and $B$  a cycle in $\dynN_{x}$.
    Define $x' = x + u(B)\cdot \chi^B$.
    Then $V_x(s) = V_{x'}(s)$, $V_x(t) = V_{x'}(t)$, and $\overline{V}_x = \overline{V}_{x'}$.
\end{lemma}

Let $f$ be a flow over time that is repeated during $[\theta_1,\theta_2]$.
Throughout the analysis of the second part of our construction it is important to lower bound how far a cycle $C$ in $\TENS{1}{\theta_2}_f$ or $\TENS{\theta_1}{\theta}_f$ can reach into the repeated part $\TENS{\theta_1}{\theta_2}_f$ of the time-expanded network.
For this purpose we consider the \emph{components} of $C$.
Let $C$ be a cycle in $\TENS{1}{\theta_2}_f$ or $\TENS{\theta_1}{\theta}_f$ or a cycle that intersects $\TENS{1}{\theta_1}_f$, $\TENS{\theta_1}{\theta_2}_f$ and $\TENS{\theta_2}{\theta}_f$.
Define $\tilde{C} = C \setminus \{\hat{s},\hat{t}\}$.
We can decompose $\tilde{C}$ into paths $\tilde{C}_1,\ldots,\tilde{C}_\ell$ as follows: 
Each of the paths $\tilde{C}_1,\ldots,\tilde{C}_\ell$ either starts at an intersection point $v$ of $\tilde{C}$ with layer $\theta_1$ or $\theta_2$, or the starting point of $\tilde{C}$ in case that $\tilde{C}$ is a path.
The endpoint of a path $\tilde{C}_j$ with $j \in \{1,\ldots,\ell\}$ starting at a node $v$ is the first intersection point $v'$ of $\tilde{C}$ with layer $\theta_1$ or $\theta_2$ following after $v$, or the endpoint of $\tilde{C}$ if there is no intersection point following after $v$. 
We call $\tilde{C}_1,\ldots,\tilde{C}_\ell$ the \emph{components} of $C$.
In the analysis of our construction we are interested in a subset of the components of $C$, called the \emph{repeated components} defined as follows:
\begin{align}\label{eq:components}
    \begin{aligned}
    \text{The \emph{repeated components of $C$}}
    \text{ are the components of $C$ that lie in $\TENS{\theta_1}{\theta_2}_f$.}
    \end{aligned}
\end{align}
In the following lemma we derive an upper bound on the number of repeated components  that a cycle $C$ can have.
\begin{lemma}\label{lem:num_of_connected_components}
    The following hold:
    \begin{itemize}
        \item If $C$ is in $\TENS{\theta_1}{\theta}_{f}$ or in $\TENS{1}{\theta_2}_{f}$ , then $C$ has at most $|V|+1$ repeated components.
        \item If $C$ intersects $\TENS{1}{\theta_1}_{f}$, $\TENS{\theta_1}{\theta_2}_{f}$ and $\TENS{\theta_2}{\theta}_{f}$, then $C$ contains at most $|V|$ repeated components starting in layer $\theta_1$ and at most $|V|$ repeated components starting in layer $\theta_2$. 
        Moreover, there is at most one additional repeated component ending in layer $\theta_1$ or ending in layer $\theta_2$.
    \end{itemize}
\end{lemma}
\begin{proof}
    By Remark~\ref{rem}$, C$ lies in $\TENS{\theta_1}{\theta}_{f}$.
    Every repeated component of $C$ with transit time at most zero is a subpath $C$ in $\TEN{\theta}_{f}$ starting in layer $\theta_2$.
    Since layer $\theta_2$ contains $|V|$ vertices, there can be at most $|V|$ such subpaths.
    Additionally $C$ also has at most one repeated component with strictly positive transit time.
    If $C$ intersects $\TENS{1}{\theta_1}_{f}$, $\TENS{\theta_1}{\theta_2}_{f}$ and $\TENS{\theta_2}{\theta}_{f}$, then by the same argument as above $C$ has at most $|V|$ repeated components starting in layer $\theta_1$ and at most $|V|$ repeated components starting in layer $\theta_2$. 
    If $C$ is of the form as depicted in Figure~\ref{fig:add_comp} there is one additional repeated component. 
    \begin{figure}[h!]
        \centering
        \includegraphics{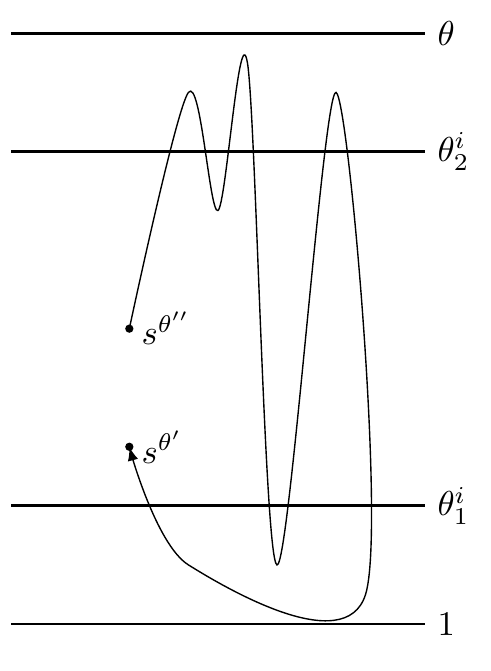}
        \caption{If $C$ has this structure, it can contain at most $2|V| + 1$ repeated components}
        \label{fig:add_comp}
    \end{figure}
\end{proof}
\section{The Projection and Lifting Lemma}\label{sec:proj_lifting}
In this section we will assume that $f$ is a given repeated maximum flow over time in $\TEN{\theta}$ that is not necessarily minimal with respect to costs.
Further, we assume that the $f$ is repeated during an interval $[\theta_1,\theta_2]$ that is ``large enough''.
More precisely, we suppose
\begin{align}\label{eq:bound_on_repeated_interval}
    \theta_2-\theta_1 > \left(\sum_{a \in A}\tau(a)\right)\cdot  \left(2\sum_{a \in A}|c(a)|u(a)+3\right)
\end{align}
In order to understand the augmentation of $f$ it will be convenient to study a projection that maps walks from $\TENS{\theta_1}{\theta_2}_f$ to $\dynN_{\phi(f)}$, and a lifting that lifts walks from $\dynN_{\phi(f)}$ back to $\TENS{\theta_1}{\theta_2}_f$.
\begin{definition}[The projection]\label{def:projection}
    We define the projection $\pi: V^\theta \rightarrow V$ as follows,
    \begin{align*}
        \pi(u) = \begin{cases}
                    &s \quad \text{ if } u = \hat{s}\\
                    &t \quad \text{ if } u = \hat{t}\\
                    &v \quad \text{ if there exists } \theta' \in \{1,\ldots,\theta\} \text{ with } v^{\theta'} = u.
                \end{cases} 
    \end{align*}
    Let $W = (v_1,a_1,\ldots,a_{k-1},v_k)$ be a walk in $\TENS{\theta_1}{\theta_2}_f$. 
    We define 

    $$\pi(W) := (\phi(v_1),(\phi(\tail(a_1)),\phi(\head(a_1)),\ldots,(\phi(\tail(a_{k-1})),\phi(\head(a_{k-1})),\phi(v_k)).$$
\end{definition}
\begin{definition}[The lifting]\label{def:lifting}
    Let $W = (v_1,a_1,\ldots,a_{\ell-1},v_{\ell})$ be a walk in $\dynN_{\phi(f)}$ and $\theta ' \in [\theta_1,\theta_2-h_{\phi(f)}(W)]$.
    Set $t_1 := \max\{\theta',\theta'-h_{\phi(f)}^-(W)\}$, $t_j := t_{j-1}+\tau(a_{j-1})$ for all $j \in \{2,\ldots,\ell\}$.
    The \emph{lifting of $W$ to layer $\theta'$} is defined as follows,
    \begin{align*}
        \lift(W,\theta') = (v_1^{t_1},a^{t_1}_1,\ldots,a^{t_{\ell-1}}_{\ell-1},v_\ell^{t_\ell}).
    \end{align*}
\end{definition}
Note that $\lift(W,\theta')$ is in $\TENS{\theta_1}{\theta_2}_f$ for all $\theta ' \in [\theta_1,\theta_2-h_{\phi(f)}(W)]$.
This follows from the regular structure of  $\TENS{\theta_1}{\theta_2}_f$, see Observation~\ref{obs:layers_are_the_same}.
The following relations are easy to verify.
\begin{obs}\label{obs:cycles}
    If $C$ is a cycle in $\TENS{\theta_1}{\theta_2}_f$, then $\pi(C)$ is a closed walk in $\dynN_{\phi(f)}$. 
    If $P$ is a path in $\TENS{\theta_1}{\theta_2}_f$, then $\pi(P)$ is a non-closed walk in $\dynN_{\phi(f)}$.

    Let $W$ be a walk in $\dynN_{\phi(f)}$ and $Q$  a walk in $\TENS{\theta_1}{\theta_2}$.
    Then $\pi(\lift(W,\theta')) = W$ for all $\theta' \in [\theta_1,\theta_2-h_{\phi(f)}(W)]$, and $Q = \lift(\pi(Q),\low(Q))$.
\end{obs}
We are now ready to state and prove the projection and the lifting lemma.
\begin{lemma}[Projection Lemma]\label{lem:projection_lemma}
    Let $C$ be a cycle in $\TENS{\theta_1}{\theta_2}_f$ with $c(C)<0$.
    Then there exists a cycle $B \subseteq \N_{\phi(f)}$ in the decomposition of $\pi(C)$ with $c(B)<0$ such that the following hold:
    \begin{enumerate}[label = (\alph*)]
        \item If $\tau(C) = 0$ \emph{and} the decomposition of $\pi(C)$ does not contain a cycle $B' \subseteq \N_{\phi(s)}$ with $\tau(B') < 0$, then $\tau(B) = 0$.
        \item If $\hat{s} \in V(C)$ or $\hat{t} \in V(C)$, respectively, then every cycle in the decomposition of $\pi(C)$ is $s$- or $t$-reachable, respectively.
        In particular, $B$ is $s$- or $t$-reachable.

        \item Suppose there is an arc $a \in A(C)$ with $a \in \cev{A}^\theta$ that is induced by an \emph{$\hat{s}$-$\hat{t}$ path} along which $f$ sends flow.
        If $\tau(C) \subseteq \dynN_{\phi(f)}(s)$, then every cycle in the decomposition of $\phi(C)$ is $s$-reachable.
        In particular, $B$ is $s$-reachable.
        Similarly, if $\tau(C) \subseteq \dynN_{\phi(f)}(t)$, then every cycle in the decomposition of $\pi(C)$ is $t$-reachable.
    \end{enumerate}
\end{lemma}
%

\begin{proof}
    Let $C = (v_1,a_1,\ldots,a_{k-1},v_k)$ be a cycle in $\TENS{\theta_1}{\theta_2}_f$.
    By Observation~\ref{obs:cycles}, $\pi(C)$ is a closed walk in $\dynN_{\phi(f)}$ that can be decomposed into cycles $B_1,\ldots,B_\ell$ in $\dynN_{\phi(f)}$ by Lemma~\ref{lem:decomposition_of_walk}.
    Note that $0< c(C) = \sum_{i = 1}^\ell c(B_i)$.
    Thus, there is an $j \in \{1,\ldots,\ell\}$ such that $c(B_j) < 0$.

    We proceed by proving part $(a)$ of Lemma~\ref{lem:projection_lemma}.
    If the decomposition of $\pi(C)$ does \emph{not} contain a cycle $B'$ with $\tau(B') < 0$, then $0 = \tau(\pi(C)) = \sum_{i = 1}^k \tau(B_i)$.
    Since $\tau(B_i) \geq 0$ for all $i \in \{1,\ldots,k\}$ by assumption, we get $\tau(B_i) = 0$ for all $i \in \{1,\ldots,k\}$. 

    For part $(b)$ note that if $\hat{s} \in V(C)$, then, by Observation~\ref{obs:cycles}, $\pi(C)$ is a closed walk in $\dynN_{\phi(f)}$ that contains the source $s$.
    This directly implies that every cycle in the decomposition of $\pi(C)$ is $s$-reachable. 
    The same arguments imply that every cycle in the decomposition of $\pi(C)$ is $t$-reachable if $\hat{t} \in V(C)$.
    It remains to prove part $(c)$ of the lemma.
    Suppose that $\pi(C) \subseteq \dynN_{\phi(f)}(s)$.
    The other case can be shown by a symmetric argument.
    Since $\pi(C) \subseteq  \dynN_{\phi(f)}(s)$, we also have $B_i \subseteq  \dynN_{\phi(f)}(s)$ for all $i \in \{1,\ldots,\ell\}$ and thus there exists an $s$-$u$ path in $\dynN_{\phi(f)}(s)$ for every $u \in V(B_i)$ for all $i \in \{1,\ldots,\ell\}$.
    By assumption there is an arc $a\in A(C)$ that is induced by an $\hat{s}$-$\hat{t}$ path along which $f$ sends flow.
    This implies that in the static network $\dynN$ there exists and $s$-$t$ path through $\pi(\cev{a})$.
    In particular, there is an $\head(\pi(\cev{a}))$-$t$ path $\tilde{P}$ in $\dynN$.
    %
    By Lemma~\ref{lem:maxflow_induces_max_flow}, $\phi(f)$ is a maximum $s$-$t$ flow in $\dynN$ sending flow along $s$-$t$ paths and cycles in $\dynN$.
    In the residual network $\dynN_{\phi(f)}$ we can thus, starting from $\head(\pi(\cev{a}))$, traverse along $\tilde{P}$, potentially traveling along the residuals of cycles from the path and cycle decomposition of $\phi(f)$ that block $\tilde{P}$, until we reach an $s$-$t$ path $P'$ from the path and cycle decomposition of $\phi(f)$.
    Next we traverse along the reverse of $P'$ to reach $s$.
    Note that since $\phi(f)$ is a maximum $s$-$t$ flow we will eventually reach an $s$-$t$ path $P'$, at the latest at the sink $t$.
    Overall, this yields that every cycle in the decomposition of $\pi(C)$ is $s$-reachable.
\end{proof}
%
%
%
\begin{lemma}[Lifting Lemma]\label{lem:lifing_lemma}
    Let $B$ be a cycle $\dynN_{\phi(f)}$ with $c(B) < 0$. 
    \begin{enumerate}[label = (\arabic*)]
        \item If $\tau(B) = 0$, then $C_1,\ldots,C_k$ with $k = \theta_2-\theta_1-h_{\phi(f)}(B)$ and  $C_i = \lift(B,i)$ for $i \in \{\theta_1,\ldots,\theta_2-h_{\phi(f)}(B)\}$ are cycles that fulfill $\pi(C_i) = B$ and $c(C_i) = c(B)<0$ for all $i \in \{1,\ldots,k\}$.
        Also, the following hold:
        \begin{enumerate}
            \item The cycles $C_1,\ldots,C_k$ are pairwise disjoint.
            \item Augmenting $f$ along $C_1,\ldots,C_k$ yields a flow over time that is repeated during $[\theta_1+h_{\phi(f)}(B),\theta_2-h_{\phi(f)}(B)]$.
            \item We have $\phi(f)+u(B)\cdot \chi^B = \phi( f + \sum_{i = 1}^k u(C_i)\cdot\chi^{C_i})$.
        \end{enumerate}
        \item If $\tau(B) \neq 0$, $B\subseteq \dynN_{\phi(f)}(s)$ or $B\subseteq \dynN_{\phi(f)}(t)$ and $B$ is $s$- or $t$-reachable, respectively, then there exists a collection of cycles $C_1,\ldots,C_k$ in $\TENS{\theta_1}{\theta_2}_f$ with $k = h_{\phi(f)}(B)$ and $c(C_i) < 0$ for all $i \in \{1,\ldots,k\}$.
        Also, the following hold:
        \begin{enumerate}
            \item The cycles $C_1,\ldots,C_k$ are disjoint.
            \item Augmenting along $C_1,\ldots,C_k$ yields a flow over time that is repeated during $[\theta_1 + 3 \cdot \sum_{a \in A} \tau(a), \theta_2 - 3\cdot \sum_{a \in A} \tau(a)]$.
            \item We have $\phi(f)+u(B)\cdot \chi^B = \phi( f + \sum_{i = 1}^k u(C_i)\cdot\chi^{C_i})$.
        \end{enumerate}
    \end{enumerate}
\end{lemma}

\begin{proof}
Let $B = (v_1,a_1,\ldots,a_{\ell-1},v_\ell)$ with $v_1 = v_\ell$.
By potentially relabeling the nodes of $B$, we can assume that  $v_1 = \low(B)$.

We begin with part (1) of the lifting lemma, i.e., $B$ is a negative cost cycle in $\dynN_{\phi(f)}$ with $\tau(B) = 0$.
Since $\tau(B) = 0$, $C_i$ is actually a cycle for all $i \in \{1,\ldots,k\}$.
The definition of the lifting in  Definition~\ref{def:lifting} and Observation~\ref{obs:cycles} imply that  $c(C_i) = c(B)$ and $\pi(C_i) = B$ for all $i \in {1,\ldots,k}$.
For all $i \in \{2,\ldots,k\}$ we obtain $C_i$ from $C_1$ by moving the cycle $C_1$ upwards  by $i-1$ layers. 
As $C_1 = \lift(B,\theta_1)$ and $B$ is a cycle in $\dynN_{\phi(f)}$ this implies that all the $C_1,\ldots,C_k$ are pairwise disjoint.
Next we prove that augmenting along $C_1,\ldots,C_k$ yields a flow over time that is repeated during $[\theta_1+h_{\phi(f)}(B),\theta_2-h_{\phi(f)}(B)]$.
Let $a = (v_j,v_{j+1}) \in A(B)$ with $j \in \{1,\ldots,k-1\}$.
Since $v_1 = \low(B)$, we augment along the following copies of $a$,
\[
    a^{\theta_1+h_{\phi(f)}(B,v_j) + \mu} = (v_j^{\theta_1+h_{\phi(f)}(B,v_j) + \mu}, v_{j+1}^{\theta_1 +h_{\phi(f)}(B,v_j) + \tau(a) + \mu}) \quad \text{ with } \mu \in \{0, \ldots, \theta_2-\theta_1-h_{\phi(f)}(B)-1\}.
\]
In particular, we augment along all the liftings of $a$ in $\TENS{\theta_1+h_{\phi(f)}(B)}{\theta_2-h_{\phi(f)}(B)}_f$.
Since $f$ is repeated, this implies that the resulting flow over time is also repeated during $[\theta_1+h_{\phi(f)}(B), \theta_2-h_{\phi(f)}(B)]$.
It remains to prove that $\phi(f)+u(B)\cdot \chi^B = \phi( f + \sum_{i = 1}^k u(C_i)\cdot\chi^{C_i})$.
Let $a \in A(B)$. 
Then, using the definition of $\phi$ in~\eqref{eq:def_of_flow_projection}, we get
\begin{align*}
    (\phi(f)+u(B)\cdot \chi^B)(a) &= \phi(f)(a) + u(B) = f(a^{\theta_1+h_{\phi(f)}(B)}) + u(C_1) \\ 
    &= f(a^{\theta_1+h_{\phi(f)}(B)})+ \left(\sum_{i = 1}^k u(C_i)\cdot  \chi^{C_i}\right)(a^{\theta_1+h_{\phi(f)}(B)})\\
    &= \phi\left(f+\sum_{i=1}^k  u(C_i)\cdot  \chi^{C_i}\right) (a).
\end{align*}
In the third equality we exploit that $a^{\theta_1+h_{\phi(f)}(B)}$ is contained in $C_i$ for exactly one $i \in \{1,\ldots,k\}$.
If $a \in A \setminus A(B)$, we get
\begin{align*}
    (\phi(f)+u(B)\cdot \chi^B)(a) &= \phi(f)(a) = f(a^{\theta_1+h_{\phi(f)}(B)})+ \left(\sum_{i = 1}^k u(C_i)\cdot  \chi^{C_i}\right)(a^{\theta_1+h(C_1)})\\
    &= \phi\left(f+\sum_{i=1}^k  u(C_i)\cdot  \chi^{C_i}\right) (a).
\end{align*}

To prove part (2) of the lifting lemma let $B$ be an $s$-reachable negative cost cycle in $\dynN_{\phi(f)}(s)$.
The case that $B$ is a $t$-reachable cycle of negative cost can be proved by a symmetric argument.
Since $B$ is $s$-reachable, there is an $s$-$v_1$ path $P_1$  and an $v_1$-$s$ path $P_2$ in $\dynN_{\phi(f)}(s)$.
We first consider the case $\tau(B) > 0$.
Set $\theta_{\start} := \theta_1+|h^-_{\phi(f)}(P_1)|$, $r := \lfloor \nicefrac{(\theta_2-\theta_{\start}-\tau(P_1)-h^+(P_2)}{h_{\phi(f)}(B)} \rfloor-1$ and define
\begin{align*}
    W := P_1 \oplus \underbrace{B\oplus B \oplus \ldots \oplus B}_{r \text{ times}} \oplus P_2.
\end{align*}
By construction $W$ is a closed walk in $\dynN_{\phi(f)}$ that starts in $s$, traverses $P_1$, then traverses the cycle $B$ $r$ times and finally follows $P_2$ back towards $s$.
Also note that $h_{\phi(f)}(W) \leq \theta_2-h_{\phi(f)}(B)$.
Thus, we can define $k = h_{\phi(f)}(B)$ many pairwise disjoint liftings of $W$ in $\TENS{\theta_1}{\theta_2}_f$, $C'_i = \lift(W,\theta_1+i)$ with $i \in \{0,\ldots,h_{\phi(f)}(B)\}$.
Note that every $C'_i$ is a path in $\TENS{\theta_1}{\theta_2}_f$ from $s^{\theta_{\start}+i}$ to $s^{\theta''}$ with $\theta_{\start} < \theta''$.
Thus, by adding the super-source $\hat{s}$, we can complete $C'_i$ to a cycle $C_i$ in $\TENS{\theta_1}{\theta_2}_f$.
The cycles $C_1,\ldots,C_k$ are disjoint by construction and we have 
\[
    c(C_i) = c(P_1) + c(P_2) + r\cdot c(B) > -2\sum_{a \in A}|c(a)|u(a)+r\cdot c(B).
\]
With our assumption on $\theta_2-\theta_1$ in~\eqref{eq:bound_on_repeated_interval} we obtain with $|h^-_{\phi(f)}|, h^+_{\phi(f)} \leq \sum_{a \in A}\tau(a)$
\begin{align*}
    r > 2\sum_{a \in A}|c(a)|u(a)
\end{align*}
Since $c(B)\leq -1$ this implies $c(C_i)<0$ for all $i \in \{1,\ldots,k\}$.
%
%
%
By construction, when augmenting along $C_1,\ldots,C_k$, we augment along all the liftings of $B$ in $\TENS{\theta_{\start}+\tau(P_1)}{\theta_2-h^+_{\phi(f)}(P_2)-h_{\phi(f)}(B)}_f$.
The first part of this lemma shows that the resulting flow over time is repeated during 
$$[\theta_{\start} + \tau(P_1) + h_{\phi(f)}(B), \theta_2-h^+_{\phi(f)}(P_2)-2h_{\phi(f)}(B)] \supseteq  [\theta_1 + 3 \cdot \sum_{a \in A} \tau(a), \theta_2 - 3\cdot \sum_{a \in A} \tau(a)], $$
such that 
$\phi(f)+u(B)\cdot \chi^B = \phi( f + \sum_{i = 1}^k u(C_i)\cdot\chi^{C_i})$.
The case $\tau(B)<0$ can be handled by the same argument by turning around every arc in the networks and swapping the roles of $P_1$ and $P_2$.
\end{proof}
\section{Overview of a Construction}
The projection and lifting lemma allow us to present a formal description of our construction of a repeated maximum minimum-cost flow over time, see Construction~\ref{alg}.
Let $f_i$ be the flow over time computed in iteration $i$ of our construction.
We assume that $f_i$  is repeated during $[\theta_1^i,\theta_2^i]$ such that $\theta_2^i-\theta_1^i > 2J(h_i)$ with $h_i = \max\{\theta_1^i,\theta-\theta_2^i\}$. 
$J(x)$ is a function that will be formerly defined in~\eqref{eq:def_bound}.
Note that $f_0$ is a repeated flow over time by Lemma~\ref{lem:maxflow_induces_max_flow} with a repeated interval that is ``large enough'' by our assumption on $\theta$ in~\eqref{eq:bound_on_TH}.

Our construction consists of two parts which we call \hyperlink{link:Step11}{Step 1} and \hyperlink{link:Step2}{Step 2}.
The first part is further subdivided into two subparts \hyperlink{link:Step11}{Step 1.1} and \hyperlink{link:Step12}{Step 1.2}.
Step 1 heavily exploits the correspondence between negative cost cycles in $\dynN_{\phi(f_i)}$ and $\TENS{\theta_1^i}{\theta_2^i}_{f_i}$ given by the Lifting and Projection Lemma (Lemma~\ref{lem:projection_lemma} and Lemma~\ref{lem:lifing_lemma}).
Step 1.1 focuses on negative cost cycles $B$ in $\dynN_{\phi(f_i)}$ with $\tau(B) = 0$.
Consider iteration $i+1$ in which we visit Step 1.1.
We detect a cycle $B$ of negative cost with $\tau(B)=0$ and then augment along the liftings of $B$ in $\TENS{\theta_1^i}{\theta_2^i}$ according to part (1) of Lemma~\ref{lem:lifing_lemma}.
This results in a flow over time $f_{i+1}$ that is repeated during $[\theta_1^i+h_{\phi(f_i)}(B),\theta^i_2-h_{\phi(f_i)}(B)]$ with $h_{\phi(f_i)}(B) \leq \sum_{a \in A}\tau(a)$.
Let $j_1$ be the last iteration in which we apply Step 1.1.
By induction and our assumption on $\theta$ in~\eqref{eq:bound_on_TH} the flow over time $f_{j_1}$ is repeated and has the property that $\TENS{\theta_1^{j_1}}{\theta_2^{j_1}}_{f_{j_1}}$ does not contain a negative cost cycle $C$ with $\pi(C) \subseteq \overline{\dynN}_{\phi(f_{j_1})}$ (Lemma~\ref{lem1} part (a)).

Next let $i+1$ be an iteration in which Step 1.2 is applied.
In this step we proceed through augmentations according to part (2) of the Lifting Lemma (Lemma~\ref{lem:lifing_lemma}).
In particular, all the $s$-reachable or $t$-reachable cycles $B$ in $\dynN_{\phi(f_i)}$ with $c(B)<0$ are considered.
The analysis of our construction requires that we augment along an $s$- or $t$-reachable Eulerian subgraph of minimum cost first and minimum transit time secondly in every iteration of Step 1.2.
Suppose Step 1.2 terminates after iteration $j_2$.
Part (2) of Lemma~\ref{lem:lifing_lemma} implies that the flow over time $f_{j_2}$ is still repeated.

We deduce that Steps 1.1 and 1.2 of our construction are visited at most $2 \cdot \sum_{a \in A} u(a) \cdot c(a)$ many iterations, i.e., for a large enough time horizon $\theta$, in particular if $\theta$ fulfills the bound in~\eqref{eq:bound_on_TH}$,$ the flow over time $f_{j_2}$ is repeated during $[\theta_1^{j_2},\theta_2^{j_2}]_{f_{j_2}}$ (Lemma~\ref{lem1} part (b)).
We also prove that the network $\TENS{\theta_1^{j_2}}{\theta_2^{j_2}}$ does not contain any negative cost cycle (Lemma~\ref{lem1} part (b)).

In every iteration $i+1$ of our construction in which \hyperlink{link:Step2}{Step 2} is visited the flow over time $f_i$ is augmented along a negative cost cycle $C$ that lies in $\TENS{1}{\theta_2^i}_{f_i}$ or in $\TENS{\theta_1^i}{\theta}_{f_i}$. 
See Figure~\ref{fig:cycles1} for examples of such cycles.
We first show that the flow over time $f_{i+1}$ is repeated by proving that the negative cost cycle $C_{i+1}$ found in iteration $i+1$ has a bounded height (Lemma~\ref{lem:height_of_component_pos} and Lemma~\ref{lem:height_of_component_neg}) after potentially applying a compression procedure that is described in Appendix~\ref{SecA:Compression}.
In a nutshell this compression procedure removes cycles of positive cost that are traversed by $C_{i+1}$ in the repeated part of the time-expanded network such that the resulting cycle has at most the cost of $C_{i+1}$ and a bounded height.
We call a cycle $C$ that has been compressed by the compression procedure a \emph{compressed cycle}.
To make sure that the cycles along which we augment in every iteration of Step 2 cannot ``stack on top of each other'' to eventually result in a non-repeated flow over time, we perform further augmentations (see lines~\ref{alg:aug1} and~\ref{alg:aug2} of the construction).
Let $j_3$ be the last iteration in which Step 2 is visited.
The flow over time $f_{j_3}$ is a repeated maximum flow over time by construction.
To deduce that $f_{j_3}$ is of minimum cost, we prove that $\TEN{\theta}_{f_{j_3}}$ does not contain any cycle $C$ of negative cost that intersects with $\TENS{1}{\theta_1^i}_{f_i}$,  $\TENS{\theta_1^i}{\theta_2^i}_{f_i}$ \emph{and} $\TENS{\theta_1^i}{\theta_2^i}_{f_i}$ (Lemma~\ref{lem:big_pos_component} and Lemma~\ref{lem:big_neg_component}). See Figure~\ref{fig:cycles2} for visualizations of such cycles.
%
\begin{algorithm}[h!]
    \small
    \SetAlgorithmName{Construction}{}  
    \DontPrintSemicolon
    \KwIn{$\dynN = (D = (V,A),s,t,u,\tau)$, $c \in \Z^{A}$, $\theta\geq 0$}
    \KwOut{A maximum minimum-cost flow over time in $\TEN{\theta}$ with respect to $c$}
    $f_0 = $ maximum flow over time in $\TEN{\theta}$ from the algorithm of Ford and Fulkerson\;
    $i = 0$\;
    \hypertarget{link:Step11}{\tcc{Step 1.1}}
    \While{Exists a cycle $B$ in $\dynN_{\phi(f_i)}$ with $c(B)<0$ and  $\tau(B) = 0$}{
        Let $C_1,\ldots,C_k$ be the cycles in $\TENS{\theta_1^i}{\theta_2^i}_f$ corresponding to $B$ by Lemma~\ref{lem:lifing_lemma}\;
        $f_{i+1} = f_i + \sum_{j = 1}^k u(C_i)\cdot \chi^{C_i}$\;
        $i = i+1$\;
    }
    \hypertarget{link:Step12}{\tcc{Step 1.2}}
    \While{$\exists$ an $s$- or $t$-reachable cycle $B$ in $\dynN_{\phi(f_i)}$ with $c(B)<0$}{
        Let $H_i$ be an $s$- or $t$-reachable Eulerian subgraph of $\dynN_{\phi(f_i)}^{u = 1}$ with minimal cost and minimal transit time with respect to the first property\;
        Let $B_1,\ldots,B_\ell$ be the collection of arc-disjoint cycles decomposing $H_i$\; 
        $f_i^0 = f_i$\;
       \For{$j \in \{1,\ldots,\ell\}$}{
       \tcc{$f_i^j$ is repeated during $[\theta_1^{i_j},\theta_2^{i_j}]$}
            Let $C^j_1,\ldots,C^j_k$ be the cycle in $\TENS{\theta_1^{i_j}}{\theta_2^{i_j}}_f$ by Lemma~\ref{lem:lifing_lemma} corresponding to $B_j$\;
            $f^j_{i} = f^{j-1}_i + \sum_{j = 1}^k  \chi^{C_i}$\;
        }

        $f_{i+1} = f^{\ell}_i$\;
        update $\mathcal{B}$\;
        $i = i+1$\;
    }
    \hypertarget{link:Step2}{\tcc{Step 2}}
    Let $H_{i} = (\emptyset,\emptyset)$ be an empty graph\;
    \While{Exists a negative cost cycle $C$ in $(\TENS{1}{\theta_2^i}_{f_i})^{u=1}$ or in  $(\TENS{\theta_1^i}{\theta}_{f_i})^{u=1}$ }{
        Let $C_i$ be a compressed cycle of minimal cost of this form\;
        $f'_{i+1} = f_i + \chi^{C_i}$\;
        $H'_{i+1} = $ Eulerian graph induced by $(V(H_i) \cup V(C_i), A(H_i) \dot \cup A(C_i))$\;
        $K'_{i+1,1},\ldots,K'_{i+1,\ell} = $ decomposition of $H'_{i+1}$ into cycles\;
        $f_{i+1} = $ augment $f'_{i+1}$ along all $K'_{i+1,j}$ with $c(K'_{i+1,j}) = 0$ with $j \in \{1,\ldots, \ell\}$\;\label{alg:aug1}
        $H_{i+1} = $ remove every $K'_{i+1,j}$ with $c(K'_{i+1,j}) = 0$ from $H'_{i+1}$  and compress every $K'_{i+1,j}$ with $j \in \{1,\ldots,\ell\}$\;\label{alg:aug2}
        $i = i+1$\;
    }
    $f = f_{i-1}$\;
    \Return{f}
    \caption{Algorithm to compute a repeated minimum cost flow over time}
    \label{alg}
\end{algorithm}
\newpage
\section{Analysis of Step 1.1 and Step 1.2}
Suppose that iteration $j_1$ is the last iteration of Construction~\ref{alg} in which Step 1.1 is visited. Accordingly, $j_2$ is the last iteration in which Step 1.2 is visited.
Recall that we still assume that $\theta$ fulfills the bound in~\eqref{eq:bound_on_TH}.
The main result of this subsection is the following lemma.
\begin{lemma}\label{lem1}
    The following two statements hold:
        \begin{enumerate}[label = (\alph*)]
            \item The flow over time $f_{j_1}$ is repeated during $[\theta_1^{j_1},\theta_2^{j_1}]$ and there is no negative cost cycle $C$ in $\TENS{\theta_1^{j_1}}{\theta_2^{j_1}}_{f_{j_1}}$ such that $\pi(C) \subseteq \overline{\dynN}_{\phi(f_{j_1})}$.
            \item The flow over time $f_{j_2}$ is repeated during $[\theta_1^{j_2},\theta_2^{j_2}]$ with 
          \begin{align}\label{eq:height_non_repeated}
        \theta_1^{j_2} = \theta- \theta_2^{j_2} = 6 \cdot \left(\sum_{a \in A} u(a)|c(a)| \right)
        \cdot \left(\sum_{a \in A}\tau(a) \right),
     \end{align}
     and $\TENS{\theta_1^{j_2}}{\theta_2^{j_2}}_{f_{j_2}}$ does not contain a cycle of negative cost.
        \end{enumerate}

\end{lemma}
From the first part of Lemma~\ref{lem:lifing_lemma} it follows that
 the flow over time $f_{j_1}$ is repeated.
It remains to prove the second statement of Lemma~\ref{lem1} (a).
%
%
Note that $\dynN_{\phi(f_0)}$ does not contain any cycles with negative transit time.
Let $0<i+1\leq j_1$ be an iteration in which Step 1.1 is visited.
In such an iteration we only augment along liftings of cycles $B$ in $\dynN_{\phi(f_i)}$ with $\tau(B) = 0$.
The following lemma implies that after such an augmentation the resulting residual network $\dynN_{\phi(f_{i+1})}$ still does not contain any cycle with negative transit time.
\begin{lemma}\label{lem:augmentation_along_zero_transit_time}
    Let $x$ be a static maximum $s$-$t$ flow in $\dynN$.
    Assume that ${\dynN}^{u=1}_{x}$ does not contain a cycle $B'$ with $\tau(B') < 0$.
    Let $B$ be a cycle in ${\dynN}^{u=1}_x$ with $\tau(B) = 0$ and let $x':= x+\chi^B$ be the flow obtained by augmenting $x$ along $B$.
    Then ${\dynN}_{x'}$ also does not contain a cycle $B'$ with $\tau(B') < 0$. 
\end{lemma}
\begin{proof}
    Aiming for a contradiction, suppose that ${\dynN}^{u=1}_{x'}$ does contain a cycle $B'$ with $\tau(B')<0$.
    Since ${\dynN}^{u=1}_{x}$ does not contain a cycle with negative transit time, the cycle $B'$ is induced by augmenting along $B$.

    We now consider the Eulerian graph $H$ induced by $(V(B) \cup V(B'), A(B) \dot\cup A(B'))$ by removing opposite arcs. Arcs that occur in both cycles are counted twice.
    Note that $H$ is a subgraph of ${\dynN}^{u=1}_x$.
    Since $H$ is Eulerian, it can be decomposed into cycles $B_1,\ldots, B_\ell$ in ${\dynN}^{u=1}_x$ such that $0 > \tau(B') = \tau(B') + \tau(B) = \tau(H) =  \sum_{i = 1}^{\ell} \tau(B_i) \geq 0$,
    a contradiction.
\end{proof}

\begin{lemma}\label{Cor:no_neg_transit_time_cycle}
    Let $i+1$ be an iteration in which Step 1.1 or Step 1.2 is visited.
    Then $\overline{\dynN}_{\phi(f_{i+1})}$ does not contain a cycle $B$ with $\tau(B) < 0$.
\end{lemma}
\begin{proof}
    We prove this lemma by induction. 
    For $i = 0$ the statement follows from Lemma~\ref{lem:maxflow_induces_max_flow}. 
    Let $i+1>0$ be some iteration in which Step 1.1 is visited.
    We obtain $f_{i+1}$ from $f_{i}$ by detecting a negative cost cycle $B$ in $\dynN_{\phi(f_{i})}$ with $\tau(B) = 0$ and then augmenting along the cycles $C_1,\ldots,C_k$ in $\TENS{\theta_1^{i}}{\theta^{i}_2}_{f_{i}}$ corresponding to $B$ according to part 1 of Lemma~\ref{lem:lifing_lemma}.
    In particular, the lifting lemma implies $\phi(f_{i})+u(B)\cdot \chi^B = \phi( f_{i} + \sum_{j = 1}^k u(C_j)\cdot\chi^{C_j}) = \phi(f_{i+1})$, i.e., $\phi(f_{i+1})$ is obtained by augmenting $\phi(f_{i})$ along $B$.
    Since $\tau(B) = 0$, Lemma~\ref{lem:augmentation_along_zero_transit_time} implies that $\overline{\dynN}_{\phi(f_{i+1})}$ does not contain a cycle with negative transit time.

    Next suppose that $i+1$ is an iteration in which Step 1.2 is visited.
    In this case $\dynN_{\phi(f_{i+1})}$ is obtained from $\dynN_{\phi(f_{i})}$ by augmenting along cycles in $\dynN_{\phi(f_{i})}(s)$ or  $\dynN_{\phi(f_{i})}(t)$.
    Thus, $\overline{\dynN}_{\phi(f_{i+1})} = \overline{\dynN}_{\phi(f_{i})}$.
    By our previous arguments, $\overline{\dynN}_{\phi(f_{i+1})}$ also does not contain any cycles $B$ with $\tau(B)<0.$
\end{proof}

We are now prepared to prove Lemma~\ref{lem1} part $(a)$.
\begin{proof}[Proof of Lemma~\ref{lem1} part $(a)$]
        By construction, after Step 1.1 the network $\dynN_{\phi(f_{j_1})}$ does not contain a cycle $B$ with $c(B) < 0$ and $\tau(B) = 0$.
        For the purpose of deriving a contradiction, assume that we are in an iteration $i+1$ in which Step 1.1 is visited and that $C$ is a cycle in $\TENS{\theta_1^i}{\theta_2^i}_{f_i}$ with $\pi(C)\subseteq \overline{\dynN}_{\phi(f_i)}$ and $c(C) < 0$.
        Define $W := \pi(C)$.
        Since $W \subseteq \overline{\dynN}_{\phi(f_i)}$, we have that $s,t \not \in V(W)$ and thus $\hat{s},\hat{t} \not \in V(C)$.
         This implies that $\tau(C) = 0$.
        By Lemma~\ref{Cor:no_neg_transit_time_cycle}, there is no cycle $B'$ in $\overline{\dynN}_{\phi(f_i)}$ with $\tau(B') < 0$.
        We can thus apply the second part of the Projection Lemma (Lemma~\ref{lem:projection_lemma}) to deduce that there exists a negative cost cycle $B$ in $\dynN_{\phi(f_i)}$ with $\tau(B) = 0$, a contradiction.   
\end{proof}
We next develop the tools to prove part (b) of Lemma~\ref{lem1}.
\begin{lemma}\label{lem:Step12}
    Let $i+1$ be an iteration in which Step 1.2 is visited.
    \begin{enumerate}[label = (\arabic*)]
        \item Let $C$ be a negative cost cycle in $\TENS{\theta_1^i}{\theta_2^i}_{f_i}$.
        Then there exists a cycle $B$ in the decomposition of $\pi(C)$ with $c(B) < 0$ that is $s$- or $t$-reachable.
        \item Let $H_{i+1}$ be the Eulerian graph chosen in this iteration.
        Then every cycle $B$ from the decomposition of $H_{i+1}$ into arc-disjoint cycles fulfills $c(B) \leq 0$. 
        Further there is at least one cycle $B$ in the decomposition of $H_{i+1}$ with $c(B) <0$.
    \end{enumerate}
\end{lemma}
\begin{proof}
    Let $C$ be a negative cost cycle in~$\TENS{\theta_1^i}{\theta_2^i}_{f_i}$.
    By Lemma~\ref{lem1}, the cycle $\pi(C)$ is not contained in $\overline{\dynN}_{\phi(f_{i})}$.
    Thus, by Observation~\ref{obs:disjoint}, the cycle $\pi(C)$ either lies in $\dynN_{\phi(j_i)}(s)$ or $\dynN_{\phi(j_i)}(t)$.

    Suppose that $\hat{s},\hat{t} \not \in V(C)$.
    Then $\tau(C) = 0$.
    Suppose that the decomposition of $\pi(C)$ does not contain a cycle $B$ with $\tau(B) < 0$.
    Then the first part of the Projection Lemma (Lemma~\ref{lem:projection_lemma}) implies that there is a negative cost cycle $B$ in $\dynN_{\phi(f_i)}$ such that $\tau(B) = 0$.
    However, after Step 1.1 such a cycle $B$ does not exist in $\dynN_{\phi(f_i)}$, a contradiction.
    Thus, there exists a cycle $B$ in the decomposition of $\pi(C)$ with $\tau(B)<0$ and part (c) of the Projection Lemma (Lemma~\ref{lem:projection_lemma}) implies that all the cycles in the decomposition of $\pi(C)$ are either $s$- or $t$-reachable.
    If $\hat{s} \in V(C)$ or $\hat{t} \in V(C)$, then part (b) of the Projection Lemma (Lemma~\ref{lem:projection_lemma}) implies that all the cycles in the decomposition of $\pi(C)$ are either $s$- or $t$-reachable.
    This finishes the proof of part (1) of this lemma.

    We proceed by proving part (2).
   If $c(B) >0$, then we can remove $B$ from $H_{i+1}$ which would yield an Eulerian subgraph with less cost, a contradiction.
    To prove that there is a cycle $B$ in the decomposition of $H_{i+1}$ with $c(B) <0$ note that if Step 1.2 is visited in iteration $i+1$, then there exists an $s$- or $t$-reachable cycle of negative cost in $\dynN_{\phi(f_i)}$. 
    This yields that $c(H_{i+1})<0$ and thus there exists a cycle $B$  in the decomposition of $H_{i+1}$ with $c(B) <0$.   
\end{proof}



We are now ready to give a proof of part $(b)$ of Lemma~\ref{lem1}.
\begin{proof}[Proof of Lemma~\ref{lem1} part $(b)$]
    With the purpose of deriving a contradiction suppose that $C$ is a negative cost cycle in $\TENS{\theta_1^{j_2}}{\theta_2^{j_2}}_{f_{j_2}}$.
    By Lemma~\ref{lem:Step12}, the decomposition of the walk $\pi(C)$ in $\dynN_{\phi(f_i)}$ according to Lemma~\ref{lem:decomposition_of_walk} contains an $s$- or $t$-reachable cycle $B$ with $c(B)<0$.
    This contradicts the fact that by construction the network $\dynN_{\phi(f_i)}$ does not contain such a cycle.
    Thus, after Step 1.2 the network  $\TENS{\theta_1^{j_2}}{\theta_2^{j_2}}_{f_{j_2}}$ does not contain a negative cost cycle.

    %
    By the second part of Lemma~\ref{lem:lifing_lemma} and Lemma~\ref{lem1}, the flow over time $f_{j_2}$ is repeated.
    Every iteration $i+1$ in which Step 1.1 or Step 1.2 is visited corresponds to at least one augmentation along a negative cost cycle in $\dynN_{\phi(f_{i})}$.
    With Lemma~\ref{lem:Step12} part (2) such an augmentation reduces the cost of the flow $\phi(f_{i})$ by at least $1$.
    The cost of a maximum minimum-cost flow in $\dynN$ is in $[-\sum_{a \in A} |c(a)| u(a), \sum_{a \in A} |c(a)| u(a)]$.
    Thus, Step 1 is visited at most $2\sum_{a \in A} |c(a)| u(a)$ times.
    The maximum flow over time $f_0$ is repeated during $[\theta_1^0,\theta_2^0]$ with $\theta^0_1 = \theta- \theta_2^0 = \sum_{a \in A}\tau(a)$.
    In the worst case an iteration $i+1$ in which Step 1 is visited increases $\theta_1^{i}$ by $3\sum_{a \in A} \tau(a)$ and decreases $\theta_2^{i}$ by $3\sum_{a \in A} \tau(a)$ (Lemma~\ref{lem:lifing_lemma}).
    Hence, $f_{j_2}$ is repeated during $[\theta_1^{j_2},\theta_2^{j_2}]$ with 
    $$\theta_1^{j_2} = \theta- \theta_2^{j_2} = 6 \cdot \left(\sum_{a \in A} u(a)|c(a)| \right)
 \cdot \left(\sum_{a \in A}\tau(a) \right).$$
 Note that if $\theta$ fulfills the bound in~\eqref{eq:bound_on_TH},  $[\theta_1^{j_2},\theta_2^{j_2}]$ is actually a repeated interval that lies within $[1,\theta]$.
 \end{proof}
We complete this section by presenting  properties that all flows over time computed during Steps 1.1 and 1.2 maintain.
These properties will play a key role for the analysis of Step 2.

By $\cut(\dynN)$ we denote all the arcs of $\dynN$ that are contained in at least one minimum $s$-$t$ cut of $\dynN$ and we define $\lift(\cut(\dynN)) := \{a^{\theta'}: \theta'
 \in \{1,\ldots,\theta\} \text{ and } a \in \cut(\dynN)\}$.
Suppose that $j_2$ is the last iteration in which Step 1.2 is visited.
\begin{lemma}\label{lem:properties}
    \begin{enumerate}[label = (\alph*)]
    The following properties hold while performing Step 1.1 and Step 1.2 of our construction.
        \item Let $a \in \cut(\dynN)$ such that $a^{\theta'}$ with $\theta' \in [1,\theta)$ is not fully congested in $f_i$ for some $i \in \{0,\ldots,j_2\}$ and suppose that $a \in \cut(\dynN)$.
        %
        Then $a^{\theta'}$ is not fully congested in $f_i$ for all $i \in \{0,\ldots,j_2\}$.
        \item The networks $\dynN_{\phi(f_{j_2})}(s)$ and $\dynN_{\phi(f_{j_2})}(t)$ do not contain an $s$- or $t$-reachable cycle $B$ with $\tau(B)<0$ and $c(B) = 0$.   
    \end{enumerate}
\end{lemma}
\begin{proof}
    To prove part (a) of this lemma, we at first need to deduce that every cycle $C$ along which we augment in Step 1.1 or Step 1.2 of Construction~\ref{alg} fulfills $A(C) \cap \lift(\cut(\dynN))= \emptyset$.
    To see this, let $i+1$ be an iteration of our construction in which we visit Step 1.1.
    Every negative cost cycle in $\TEN{\theta}_{f_i}$ along which we augment in such an iteration is a lifting of a cycle $B$ in $\dynN_{\phi(f_i)}$ which does not intersect $\cut(\dynN)$.
    Thus, $A(C) \cap \lift(\cut(\dynN)) = \emptyset$.
    In an iteration $i+1$ in which we visit Step 1.2 an augmenting cycle $C$ projects down to a closed walk $W$ with two possible forms (1) or (2):
    (1) $W$ starts at $s$, then follows a path $P_1$ towards a cycle $B \subset \dynN_{\phi(f_i)}(s)$ , traverses $B$ multiple times until it follows a path $P_2$ back towards $s$, or (2) $W$ starts at $t$, then follows a path $P_1$ towards a cycle $B \subset \dynN_{\phi(f_i)}(t)$ , traverses $B$ multiple times until it follows a path $P_2$ back towards $t$.
    We only consider case (1).
    The proof for case (2) is analogue.
    We clearly have $A(B) \cap \cut(\dynN) = \emptyset$. 
    Both paths $P_1$ and $P_2$ start and end in $\dynN_{\phi(f_i)}(s)$ and since $\phi(f_i)$ is a maximum $s$-$t$ flow, the paths cannot cross a minimum $s$-$t$ cut in forward direction. 
    Thus, $P_1,P_2 \subseteq \dynN_{\phi(f_i)}(s)$ and hence $A(C) \cap \lift(\cut(\dynN)) = \emptyset$.

    To deduce part (a) of Lemma~\ref{lem:properties}, note that 
    for every $i \in \{0,\ldots,j_2\}$ the flow $\phi(f_i)$ is a maximum $s$-$t$ flow by Lemma~\ref{lem:maxflow_induces_max_flow} and the fact that for every $i \in \{1,\ldots,j_2\}$ is obtained from $\phi(f_0)$ by a sequence of augmentations along cycles (Lemma~\ref{lem:lifing_lemma}).
    In particular, $a^{\theta'}$ is fully congested in $f_i$ for all $\theta' \in [\theta_1^i,\theta_2^i)$.
    Thus, the fact that $a^{\theta'}$ is not fully congested implies that $\theta' \in [1,\theta^{i}_1)$ or $\theta' \in [\theta^{i}_2,\theta)$.
    For every $i \in \{0,\ldots,j_2-1\}$, the cycle $C$ along which we augment in iteration $i+1$ lies in $\TENS{\theta_1^{i}}{\theta_2^{i}}_{f_{i}}$. 
    Thus, an arc $a^{\theta'}$ with  $\theta' \in [1,\theta^{i}_1)$ or $\theta' \in [\theta^{i}_2,\theta)$ that is not fully congested in $f_{i}$ will stay non-congested in the following iterations until iteration $j_2$.
    Since the negative cost cycles along which we augment in iterations $\{1,\ldots,j_2\}$ are disjoint from $\lift(\cut(\dynN))$, an arc $a^{\theta'}$ with $a \in \cut(\dynN)$ and $\theta' \in [1,\theta^{i}_1)$ or $\theta' \in [\theta^{i}_2,\theta)$  that is not fully congested in $f_{i}$ with $i \in \{0,\ldots,j_2\}$ is also not fully congested in $f_\mu$ with $\mu \in \{0,\ldots,i\}$.

    To show part (b), suppose $B$ is an $s$-reachable cycle in $\dynN^{u=1}_{\phi(f_{j_2})}(s)$ with $\tau(B)<0$ and $c(B) = 0$ and that $B$ is induced by iteration $\ell$ of the construction, i.e., $B$ is not contained in $\dynN_{\phi(f_{i})}(s)$ for all $0 \leq i<\ell$.
    Note that $\ell\geq 1$, as the network $\dynN_{\phi(f_0)}$ does not contain any cycles with negative transit time by Lemma~\ref{lem:maxflow_induces_max_flow}.
    Let $H_\ell$ be the Eulerian subgraph of $\dynN_{\phi{f_{\ell-1}}}^{u=1}$ along which we augment in iteration $\ell$ of the construction.
    Consider the Eulerian graph  $H$ which is induced by $(V(H_\ell) \cup V(B), A(H_\ell) \dot\cup A(B))$ by removing opposite arcs.
    Arcs that occur in $A(H_\ell)$ and $A(B)$ are counted twice. 
    Note that every subgraph of $H$ is a subgraph of $\dynN^{u=1}_{\phi(f_{\ell-1})}$ with $c(H) = c(H_\ell) + c(B) = c(H_\ell)$ and $\tau(H) = \tau(H_\ell) + \tau(B) < \tau(H_\ell)$.
    This contradicts the fact that in iteration $\ell$ we have chosen an Eulerian subgraph $H_\ell$ of $\dynN^{u=1}_{\phi_{\ell-1}}$ with minimal cost and among all of them minimal transit time.
\end{proof}

\section{Analysis of Step 2}
Suppose that Step 1.2 finishes after iteration $j_2$.
The starting point of our analysis of Step 2 is the flow over time $f_{j_2}$ that is repeated during $[\theta_1^{j_2},\theta_2^{j_2}]$ by Lemma~\ref{lem1} part $(b)$.
We prove inductively that after every iteration of the while-loop in Step 2 the resulting flow over time is still repeated.

Assuming that $f_i$ is repeated during $[\theta_1^i,\theta_2^i]$with $\theta_2^i-\theta_1^i > 2J(h_i)$ with $i \geq j_2$ and $h_i = \max\{\theta_1^i,\theta-\theta_2^i\}$ define
\begin{align}\label{eq:def_h}
    h := 2\left(\sum_{a \in A}\tau(a) \right) \cdot \left(1 + \sum_{a \in A}|c(a)|u(a)\right),
\end{align}
 and  the function $J: \R \rightarrow \R$ by
\begin{align}\label{eq:def_bound}
    J(x) := \left(\sum_{a \in A}\tau(a)\right)\left( (|V|+1)\left(2\sum_{a \in A}\tau(a) +3 \right)+1+2h+2x  \right) \cdot \sum_{a \in A}|c(a)|u(a) + 3\sum_{a \in A}\tau(a) 
\end{align}
for all $x \in \R$.
\begin{lemma}\label{lem:Step2_repeated}
    Let $i+1$ be an iteration of our construction in which Step 2 is visited.
    Then the flow over time $f_{i+1}$ is repeated during $[\theta_1^{i+1},\theta_2^{i+1}]$ with $\theta_1^{i+1} \leq \theta_1^{j_2}+J(\theta_1^{j_2})$ and $\theta_2^{i+1} \geq \theta_2^{j_2} - J(\theta_1^{j_2})$  if $\theta >2((\theta_1^{j_2}+ J(\theta_1^{j_2}) + J(\theta_1^{j_2}+ J(\theta_1^{j_2}))) $.
\end{lemma}
%
%
%
%
Let $i+1$ be an iteration of our construction in which Step 2 is applied, suppose that $f_i$ is repeated during $[\theta_1^i,\theta_2^i]$with $\theta_2^i-\theta_1^i > 2J(h_i)$ with $h_i = \max\{\theta_1^i,\theta-\theta_2^i\}$, and let $C$ be a negative cost cycle in $\TENS{1}{\theta}_{f_{i}}$ or $\TENS{\theta_1^i}{\theta}_{f_{i}}$.
%
%
%
%
Recall the definition of a \emph{repeated component} of $C$ in~\eqref{eq:components}.
In order to prove that after augmenting along a negative cost cycle $C$ in $\TENS{1}{\theta_2^i}_{f_i}$ or in $\TENS{\theta_1^i}{\theta}_{f_i}$  the resulting flow over time is still repeated, we show that the height of each repeated component is bounded after potentially applying the compression procedure based on Appendix~\ref{SecA:Compression}.
A priori it is unclear how to bound the height of a repeated component of $C$.\
In fact our analysis is based on a connection between cost and height of these repeated components.
Deriving this connection requires us to analyze structural properties of the repeated components of $C$.
For this we need to consider $W_j := \pi(C_j')$ for all $j \in \{1,\ldots,k\}$, i.e., the projection of each $C_j'$ to $\dynN_{\phi(f_i)}$, and make a case distinction on the transit time of $W_j$.
If $\tau(W_j) > 0$, we deduce that $W_j$ is contained in either $\dynN_{\phi(f_i)}(s)$ or $\dynN_{\phi(f_i)}(t)$ which directly implies a bound on the cost of $W_j$ (Lemma~\ref{lem:no_cycle_crossing_cuts_pos}) by the structural properties of $\dynN_{\phi(f_i)}$ we have deduced before.
If $\tau(W_j) = 0$ the same statement holds under the additional assumption that the height of $W$ is larger than a lower bound (Lemma~\ref{lem:no_cycle_crossing_cuts_0}).
This allows us derive that the cost of $W_j$ is strictly positive if the height of $W_j$ exceeds a certain bound (Lemma~\ref{lem:cost_of_component_0}).
If $\tau(W_j) < 0$, a more precise analysis of the structure of $W_j$ (Lemma~\ref{lem:structure_of_connected_component_neg}) allows us to derive a lower bound on the cost of $W_j$ (Lemma~\ref{lem:cost_of_component_neg}).
Overall, this yields a lower bound on the cost of $C$ (Lemma~\ref{eq:bound_on_cost_overall}).
Putting together this lower bound and the structural properties derived for the repeated components of $C$ enables us to bound the height for each repeated component of $C$ (Lemma~\ref{lem:height_of_component_neg} and Lemma~\ref{lem:height_of_component_pos}).

To show that $W_j$ is contained in either $\dynN_{\phi(f_i)}(s)$ or $\dynN_{\phi(f_i)}(t)$ if $\tau(W_j) \geq 0$ (and the height of $W_j$ exceeds a lower bound), the following three lemmas are central.
\begin{lemma}\label{lem:path_towards_cut}
    Let $i\geq j_2$ be an iteration of our construction and suppose that $f_i$ is repeated during $[\theta_1^i,\theta_2^i]$.
    Let $a \in \cut(\dynN)$  with $\theta' \in [1,\theta)$ such that $a^{\theta'}$  is not fully congested in the flow over time $f_{i}$.
     If $\theta' \in [\theta^i_2,\theta)$, then there is an $\hat{s}$--$\tail(a^{\theta'})$ path in $\TEN{\theta}_{f_i}$.
    If $\theta' \in [1,\theta^i_1)$, then there is a $\head(a^{\theta'})$--$\hat{t}$ path in $\TEN{\theta}_{f_i}$.
\end{lemma}
\begin{proof}
    We only give a proof for the case $\theta' \in [\theta_2^i,\theta)$.
    The proof for the other case is symmetric.
    We distinguish two subcases. 
    First suppose that the arc $a^{\theta'}$ is not fully congested in the flow over time $f_{j_2}$.
    In this case, Lemma~\ref{lem:properties} part (a) yields that the arc $a^{\theta'}$ is not fully congested in $f_0$.
    By the algorithm of Ford and Fulkerson, the flow that is sent through liftings of $a$ via $f_0$ is induced by a set of $s$-$t$ paths $P_1,\ldots,P_k$ in $\dynN$.
    More precisely, flow is sent into the path $P_i$ during $\{1,\ldots,\theta-\tau(P_i)\}$ for all $i \in \{1,\ldots,k\}$, i.e., flow covers the arc $a$ via $P_i$ during $\{\tau(P_i[s,\head(a)],\ldots, \theta-\tau(P_i)+\tau(P_i[s,\head(a)])-1\}$.
    Without loss of generality suppose that $P_1 = \argmin\{\theta-\tau(P)+\tau(P[s,\head(a)]-1): P \in \{P_1,\ldots,P_k\}\}$.
    Thus, for all $\theta' \in \{\theta-\tau(P_1)+\tau(P_1[s,\head(a)]),\ldots,\theta\}$, the arc $a^{\theta'}$ is not fully congested in $f_0$.
    For $\mu = \theta'-\tau(P_1[s,\head(a)])$, denote by $P^\mu_1[s,\head(a)]$ the lifting of $P_1[s,\head(a)]$ that starts at $s^{\mu}$.
    Recall that $\theta-\theta_2^0,\theta_1^0 < \sum_{a \in A}\tau(a)$ by Lemma~\ref{lem:maxflow_induces_max_flow}.
    Together with our assumption on $\theta$ in~\eqref{eq:bound_on_TH} this yields $\mu \geq \theta_2^0 - \sum_{a \in A}\tau(a) \geq \theta-2\sum_{a \in A}\tau(A) > \theta_1^0$.
    Clearly, $P^\mu_1[s,\head(a)]$ ends with the arc $a^{\theta'}$.
    It remains to prove that $P^\mu_1[s,\head(a)]$ is a path in $\TEN{\theta}_{f_0}$.
    Recall that in $f_0$ we stop sending flow into $P_1$ at time $\theta-\tau(P)+1$.
    Since $\mu \geq \theta-\tau(P)+1$ this implies that $P^\mu_1[s,\head(a)]$ is a path in $\TEN{\theta}_{f_0}$.
    In particular, there is an $\hat{s}$--$\tail(a^{\theta'})$ path in $\TEN{\theta}_{f_0}$.
    Augmenting along cycles in the following iterations does not change reachability relations.
    Thus, there is an $\hat{s}$--$\tail(a^{\theta'})$ path in $\TEN{\theta}_{f_i}$.

    As a second case suppose that $a^{\theta'}$ is fully congested in $f_{j_2}$ but not fully congested in $f_i$.
    We prove this case by induction. 
    First let $i = j_2+1$. 
    Thus, there is a cycle  $C$ along which we augment in iteration $j_2+1$ with the property that it contains the backward arc of $a^{\theta'}$ and another arc $b^{\theta''}$ such that $b \in \cut(\dynN)$ and $\theta'' \in [\theta_2^i,\theta)$.
    In particular $b^{\theta''}$ is not fully congested in $f_{j_2}$.
    The arguments presented in the the first subcase show that there exists an $\hat{s}$-$\head(b^{\theta''})$ path $P$ in $\TEN{\theta}_{f_{j_2}}$.
    In  $\TEN{\theta}_{f_{j_2+1}}$ we can now follow along the path $P$ until we reach a blocked arc, and then follow along the cycle $C$ in backward direction until we reach $\head(a^{\theta''})$.
    This results in a $\hat{s}$-$\head(a^{\theta''})$ path $P'$ in $\TEN{\theta}_{f_{j_2+1}}$.
    %
    Next, let $i>j_2+1$.
    If $a^{\theta''}$ is not fully congested in $f_{i-1}$, then from the hypothesis of induction there exists an $\hat{s}$-$\head(a^{\theta'})$ path $P$ in $\TEN{\theta}_{f_{i-1}}$.
    Again, the path $P$ can be combined with the augmenting cycle $C$ in iteration $i$ to an $\hat{s}$-$\head(a^{\theta'})$ path $P'$ in $\TEN{\theta}_{f_{i}}$.

    If $a^{\theta''}$ is fully congested in $f_{i-1}$, then the same arguments as for $i=j_2+1$ yield that there exists an $\hat{s}$-$\head(a^{\theta'})$ path $P'$ in $\TEN{\theta}_{f_{i}}$.
\end{proof}


%
We proceed by proving two additional structural lemmas that are useful during the analysis of Step 2.
\begin{lemma}\label{lem:cycles_are_reachable}
    Let $i \geq j_2$ be an iteration of our construction, suppose that $f_i$ is repeated during $[\theta_1^i,\theta_2^i]$ and let $C$ be a negative cost cycle in  $(\TENS{1}{\theta})_{f_i}^{u=1}$.
    Let $C'$ be a repeated component of $C$ and $W := \pi(C')$.
    Then every cycle $B$ in the decomposition of $W$ according to Lemma~\ref{lem:decomposition_of_walk} with $B \subseteq \dynN_{\phi(f_i)}(s)$ or $B \subseteq \dynN_{\phi(f_i)}(t)$ is $s$- or $t$-reachable, respectively.
\end{lemma}
\begin{proof}
    If there is an arc $a \in A(B)$ with $a \in \cev{A}$ such that a lifting of this arc in $C$ is induced by a $\hat{s}$-$\hat{t}$ path along which $f$ sends flow, then the statement of the lemma follows directly from part $(c)$ of the Projection Lemma (Lemma~\ref{lem:projection_lemma}).

    To prove the lemma when $A(B)$ does not contain such an arc, we distinguish whether $B$ lies in a repeated component with transit time zero or not.
    Suppose that $\tau(W) = 0$ and $B \subseteq \dynN_{\phi(f_i)}(s)$.
    The case  $B \subseteq \dynN_{\phi(f_i)}(t)$ can be shown by a symmetric argument.
    By definition of $\dynN_{\phi(f_i)}(s)$ there is an $s$-$v$ path in $\dynN_{\phi(f_i)}$ for all $v \in V(B)$.
    As $h^\theta_{f_i}(C') > 0$ and $\tau(W) = 0$, there is an arc $a \in A(C'_1)$ with $\tau(\pi(a)) < 0$ that occurs in $W$ after $B$.
    Thus, $a$ is induced by an $\hat{s}$-$\hat{t}$ path along which $f_i$ sends flow.
    Suppose that $a$ is the first such arc that occurs in $C'$ after a lifting of $B$.
    Let $P$ be an $\hat{s}$-$\hat{t}$ path along which $f_i$ sends flow that passes through $\cev{a}$.
    Starting from a vertex in a lifting of $B$ we can now follow $C'$ in $\TEN{\theta}$ towards $\tail(a)$.
    All the arcs on this path $P'$ are either induced by a cycle $C_0$ that lies completely in one layer of $\TEN{\theta}$ or are arcs in $\TEN{\theta}$.
    If we reach an arc of the first kind, we can thus follow along the inducing cycle $C_0$ and then proceed along $C'$.
    Overall, it is thus possible to reach $\tail(a)$ in $\TEN{\theta}$ from a node in a lifting of $B$.
    Once we have reached $\tail(a)$, we follow $P$ towards $\hat{t}$.
    This construction implies that we can reach $\hat{t}$ from a node in a lifting of $B$ in $\TEN{\theta}$.
    In particular, $t$ can be reached from $V(B)$ in $\dynN$.
    Thus, to reach $s$ from $V(B)$ in $\dynN_{\phi(f_i)}$ we can follow a path from $V(B)$ towards $t$, potentially walking backward along cycles from the decomposition of $\phi(f_i)$ until we reach an $s$-$t$ path from the decomposition of $\phi(f_i)$.
    Then we can follow this path in backward direction to reach $s$.

    Next suppose that $\tau(W) \neq 0$.
    Note that in this case $C$ either contains $\hat{s}$ or $\hat{t}$ or $C$ intersects $\TENS{1}{\theta_2^i}$ \emph{and} $\TENS{\theta_1^i}{\theta_2^i}$ \emph{and} $\TENS{\theta_2^i}{\theta}$.
    If $C$ intersects $\TENS{1}{\theta_2^i}$ \emph{and} $\TENS{\theta_1^i}{\theta_2^i}$ \emph{and} $\TENS{\theta_2^i}{\theta}$ and does not contain $\hat{s}$ or $\hat{t}$, then there is some arc $a \in C$ with $\tau(\pi(a))<0$ which is induced by a path along which $f_i$ sends flow.
    Thus, the statement of the lemma follows from the same arguments as presented before.
    Now suppose that $C$ contains $\hat{s}$ or $\hat{t}$.
    If there is an arc $a \in A(C)$ that occurs after a lifting of $B$ in $C$ that is induced by a path along which $f_i$ sends flow, then the statement of the lemma follows by the same arguments as presented before.
    Suppose such an arc does not exist.
    If $C$ contains $\hat{s}$, then, we can walk towards $\hat{s}$ in $\TEN{\theta}$ starting from a node in a lifting of $B$ and then following $C$ while potentially diverting around cycles $C_0$ that lie completely in one layer of $\TEN{\theta}$.
    Once we have reached $\hat{s}$, we can follow any path towards $\hat{t}$ in $\TEN{\theta}$.
    If $C$ contains $\hat{t}$, then we can directly reach $\hat{t}$ from a node in a lifting of $B$ by potentially diverting around cycles $C_0$ that lie completely in one layer of $\TEN{\theta}$.
    Overall, this implies that it is possible to reach $\hat{t}$ from a node in a lifting of $B$ in $\TEN{\theta}$.
    Similar arguments show that $s$ can also be reached from $B$ in $\dynN_{\phi(f_i)}$.

\end{proof}
\begin{lemma}\label{lem:augmenting_path}
    Let $C'$ be a repeated component of $C$ and $W = \pi(C')$.
    The following statements hold.
    \begin{itemize}
        \item  If $C$ contains $\hat{s}$, then $C'' = C \cap \TENS{1}{\theta_2^i}_{f_i} = \emptyset$.
        Similarly, if $C$ contains $\hat{t}$, then $C'' = C \cap \TENS{\theta_2^i}{\theta}_{f_i} = \emptyset$.
        \item  If there exists a node $v \in W$ with $v \in \dynN_{\phi(f_i)}(s)$ such that $h_{\phi(f_i)}(v,W)\geq \sum_{a \in A} \tau(a)$, then $\hat{t} \not \in C$ , $C'' = C \cap \TENS{\theta_2^i}{\theta}_{f_i} = \emptyset$ and there is no node $v' \in V$  with $v' \in \dynN_{\phi(f_i)}(t)$ such that $h_{\phi(f_i)}(v',W)\geq \sum_{a \in A} \tau(a)+h_{\phi(f_i)}(v,W)$

    Similarly, if there exists a node $v \in W$ with $v \in \dynN_{\phi(f_i)}(t)$ such that $h_{\phi(f_i)}(v,W) \sum_{a \in A} \tau(a)$, then $\hat{s} \not \in C$, $C'' = C \cap \TENS{1}{\theta_1^i}_{f_i} = \emptyset$  and there is no $v' \in W$  with $v' \in \dynN_{\phi(f_i)}(s)$ such that $h_{\phi(f_i)}(v',W) \sum_{a \in A} \tau(a)+h_{\phi(f_i)}(v,W)$. 
    \end{itemize}

\end{lemma}
\begin{proof}
    We start by proving the first part of the lemma.
    We only consider the case that $\hat{s} \in C$.
    The other case is symmetric.
    For the purpose of deriving a contradiction suppose that $C'' = C \cap \TENS{1}{\theta_2^i}_{f_i}$ contains an arc from $\lift(\cut(\dynN))$.
    In particular, there is an arc $a \in \cut(\dynN)$ and $\theta' \in [1,\theta_1^i)$ such that $a^{\theta'}$ is not fully congested in $f_i$.
    Thus, by Lemma~\ref{lem:path_towards_cut} there exists an $\head(a^{\theta'})$-$\hat{t}$ path $Q$ in $\TEN{\theta}_{f_i}$.
    Together with the subpath of $C$ starting in $\hat{s}$ and ending in $\head(a^{\theta'})$ this gives an $\hat{s}$-$\hat{t}$ augmenting path in $\TEN{\theta}$, contradicting the maximality of $f_i$.

    For the second part, we only consider the case that there is a $v \in W$ with $v \in \dynN_{\phi(f_i)}(t)$ such that $h_{\phi(f_i)}(v,W)\geq \sum_{a \in A} \tau(a)$
    The other case can be shown by the same arguments just swapping $s$ and $t$.
    By the definition of $\dynN_{\phi(f_i)}(t)$ and Lemma~\ref{lem:invariant_under_augmentations} and since $h_{\phi(f_i)}(v,W)\geq \sum_{a \in A} \tau(a)$ there is a $v$-$t$ path $P_t$ in $\dynN_{\phi(f_i)}(t)$ that can be lifted to $\TENS{\theta_1^i}{\theta_2^i}_{f_i}$ in such a way that it connects to $C$.
    If $C$ contains $\hat{s}$, this lifting of $P_t$ results in an $\hat{s}$-$\hat{t}$ augmenting path contradicting the maximality of $f_i$.
    Next assume that $C'' = C \cap \TENS{\theta_2^i}{\theta}_{f_i}$ and suppose that $C''$ contains a forward arc from $\lift(\cut(\dynN))$.
    In particular, there is an arc $a \in \cut(\dynN)$ and $\theta' \in [\theta_2^i,\theta)$ such that $a^{\theta'}$ is not fully congested in $f_i$.
    Lemma~\ref{lem:path_towards_cut} yields that there is an $\hat{s}$-$\tail(a^{\theta'})$ path $P$ in $\TENS{1}{\theta}_{f_i}$.
    Overall, the lifting of $P_t$ and the path $P$ result in an $\hat{s}$-$\hat{t}$ path, contradicting the maximality of $f_i$.

     If there were a $v' \in W$ with $v' \in \dynN_{\phi(f_i)}(s)$ such that $h_{\phi(f_i)}(v',W)\geq \sum_{a \in A} \tau(a)+h_{\phi(f_i)}(v,W)$, the $s$-$v'$ path $P_s$ in $\dynN_{\phi(f_i)}(s)$ be can be lifted to $\TENS{\theta_1^i}{\theta_2^i}_{f_i}$ in such a way that it connects to $C$ and does not interfere with the lifting of $P_t$.
    This results in an $\hat{s}$-$\hat{t}$ path, contradicting the maximality of $f_i$.

\end{proof}

\subsection{Bounding the Cost of Each Component}
Let $i+1$ be an iteration in which Step 2 is visited, suppose that $f_i$ is repeated during $[\theta_1^i,\theta_2^i]$ with $\theta_2^i-\theta_1^i > 2J(h_i)$ and $h_i = \max\{\theta_1^i,\theta-\theta_2^i\}$.
Let $C$ be a negative cost cyclic in $\TEN{\theta}_{f_i}$.
Our goal is to prove that $h^\theta_{f_i}(C) \leq J(h_i)$ with $h_i = \max\{\theta_1^i,\theta-\theta_2^i\}$ if $C$ is in $\TENS{1}{\theta_2^i}_{f_i}$ or $\TENS{\theta_1^i}{\theta}_{f_i}$ in order to conclude that $f_{i+1}$ is still repeated.

We start by deriving lower bounds on the cost of the repeated components of $C$ that will turn out useful in the analysis of their heights.
For later use we focus on general negative cost cycles in $\TEN{\theta}_{f_i}$ and do not only stick to negative cost cycles in $\TENS{1}{\theta_2^i}_{f_i}$ or in  $\TENS{\theta_1^i}{\theta}_{f_i}$.
Thus, in this subsection $C$ is a negative cost cycle in $\TEN{\theta}_{f_i}$ if not stated otherwise.
\begin{lemma}\label{lem:walk_in_Ns_or_Nt_cost}
    Let $W$ be a walk in $\dynN_{\phi(f_i)}(s)$ or $\dynN_{\phi(f_i)}(t)$.
    Then $c(W) \geq -\sum_{a \in A}|c(a)|u(a)$.
\end{lemma}
\begin{proof}
    $W$ decomposes into a path $P$ with $c(P) \geq -\sum_{a \in A}|c(a)|u(a)$ and cycles that are $s$- or $t$-reachable by Lemma~\ref{lem:cycles_are_reachable} and thus fulfill $c(B) \geq 0$ by construction.
\end{proof}
\paragraph{Components with Strictly Negative Transit Time.}

We start by proving the following lemma about the structure of a repeated component of $C$ with strictly negative transit time.
\begin{lemma}\label{lem:structure_of_connected_component_neg}
    Let $C'$ be a repeated component of $C$ such that $\tau(W) < 0$ with $W:= \pi(C')$.
    We distinguish the following cases:
    \begin{enumerate}[label = (\arabic*)]
        \item $W$ lies completely in $\dynN_{\phi(f_i)}(s)$ or $\dynN_{\phi(f_i)}(t)$
        \item $W = W(t) \oplus \overline{W} \oplus{W}(s)$ with $W(t) \subseteq \dynN_{\phi(f_i)}(t)$, $W(s) \subseteq \dynN_{\phi(f_i)}(s)$ and $\overline{W} \subseteq \overline{\dynN}_{\phi(f_i)}$, $h_{\phi(f_i)}(W(t)) < \sum_{a \in A}\tau(a)$ and $h_{\phi(f_i)}(W(t) \oplus \overline{W}) < 2\sum_{a \in A}\tau(a)$ 
        \item $W = W(t) \oplus \overline{W} \oplus{W}(s)$ with $W(t) \subseteq \dynN_{\phi(f_i)}(t)$, $W(s) \subseteq \dynN_{\phi(f_i)}(s)$ and $\overline{W} \subseteq \overline{\dynN}_{\phi(f_i)}$, $h_{\phi(f_i)}(W(s)) < \sum_{a \in A}\tau(a)$ and $h_{\phi(f_i)}(\overline{W} \oplus W(s)) < 2\sum_{a \in A}\tau(a)$
    \end{enumerate}
    We have,
    \begin{itemize}
        \item If $C'$ is in $\TENS{\theta_1^i}{\theta}_{f_i}$, then $W$ is of the form (1) or (2).
        \item If $C'$ is in $\TENS{1}{\theta_2^i}_{f_i}$, then $W$ is of the form (1) or (3).
        \item If $C'$ intersects $\TENS{1}{\theta_1^i}_{f_i}$, $\TENS{\theta_1^i}{\theta_2^i}_{f_i}$ and $\TENS{\theta_2^i}{\theta}_{f_i}$, then $W$ is of the form (1), (2) or (3).
    \end{itemize}
\end{lemma}
\begin{proof}
    Since $\phi(f_i)$ is a maximum $s$-$t$ flow by Lemma~\ref{lem:maxflow_induces_max_flow}, it is immediate that $W\cap \cut(\dynN) = \emptyset$.
    This directly implies that $C' \cap \lift(\cut(\dynN)) = \emptyset$.

    We first assume that $C'$ is in $\TENS{\theta_1^i}{\theta}_{f_i}$.
    In this case $C'$ is a path that ends in layer $\mu$ with $\theta_1^i \leq \mu < \theta_2^i$ and starts in layer $\theta_2^i$.
    In particular, $W$ ends with $s$ or $t$.
    If $W$ ends in $t$, then $W$ is contained in $\dynN_{\phi(f_i)}(t)$ as $W \cap \cut(\dynN) = \emptyset$.
    We next consider the case that $W$ ends in $s$.
    Suppose that $W$ does not lie in $\dynN_{\phi(f_i)}(s)$ and that $W$ starts in $\dynN_{\phi(f_i)}(t)$.
    We consider two subcases.
    Define $W(t) := W \cap \dynN_{\phi(f_i)}(t)$.
    As $\hat{s} \in C$, Lemma~\ref{lem:augmenting_path} implies that $h_{\phi(f_i)}(W(t)) < \sum_{a \in A} \tau(a)$.
    Define $\overline{W} := W \cap \overline{\dynN}_{\phi(f_i)}$.
    Note that $h^-_{\phi(f_i)}(\overline{W}) \geq - \sum_{a \in A}\tau(a)$ as otherwise the decomposition of $\overline{W}$ would contain a cycle $B$ with $\tau(B) < 0$ contradicting Lemma~\ref{Cor:no_neg_transit_time_cycle}.
    Thus, $h_{\phi(f_i)}(W(t) \oplus \overline{W}) <  2 \sum_{a \in A}\tau(a) $.
     If $W$ starts in $\overline{\dynN}_{\phi(f_i)}$, then we have $h_{\phi(f_i)}(W(t)) = 0$ and thus the statement of the lemma follows as above.
     Overall, $W$ can take forms (1) or (2) if $C'$ is in $\TENS{\theta_1^i}{\theta}_{f_i}$.
    If $C'$ is in $\TENS{1}{\theta_2^i}_{f_i}$, then by reversing the network the proof above shows that $W$ can take forms (1) or (3).

    Finally, we consider the case that $C$ intersects $\TENS{1}{\theta_1^i}_{f_i}$, $\TENS{\theta_1^i}{\theta_2^i}_{f_i}$ and $\TENS{\theta_2^i}{\theta}_{f_i}$.
    Suppose that $W$ does not lie completely in $\dynN_{\phi(f_i)}(s)$ or $\dynN_{\phi(f_i)}(t)$.
    With the purpose of deriving a contradiction suppose that $W = W(t) \oplus \overline{W} \oplus W(s)$ with $h_{\phi(f_i)}(W(s)), h_{\phi(f_i)}(W(t)) \geq \sum_{a \in A} \tau(a)$.
    Since $\theta_2^i-\theta_1^i > 2J(h_i)$ by assumption, we can identify $v \in W(t)$ with $h(v,W) \geq \sum_{a \in A}\tau(A)$ and  $v' \in W(s)$ with $h_{\phi(f_i)}(v',W) \geq h_{\phi(f_i)}(v,W)+ \sum_{a \in A}\tau(A)$.
    This is a contradiction by Lemma~\ref{lem:augmenting_path}.
    Thus, we have $h_{\phi(f_i)}(W(s)) < \sum_{a \in A} \tau(a)$ or $h_{\phi(f_i)}(W(t)) < \sum_{a \in A} \tau(a)$.
    If $h_{\phi(f_i)}(W(t)) < \sum_{a \in A} \tau(a)$, we can proceed as for the case that $C'$ is in $\TENS{\theta_1^i}{\theta}_{f_i}$ to deduce that $W$ is of the form (2).
    Similarly, if $h_{\phi(f_i)}(W(s)) < \sum_{a \in A} \tau(a)$, we can proceed as for the case that $C'$ is in $\TENS{1}{\theta_2^i}_{f_i}$ to deduce that $W$ is of the form (3).

\end{proof}
\begin{rem}\label{rem}
    In the following we will occasionally have to distinguish whether $C'$ is in $\TENS{\theta_1^i}{\theta}_{f_i}$ or $\TENS{1}{\theta_2^{i}}_{f_i}$.
    We will then consider only one of the cases. 
    This is without loss of generality by reversing the network.
\end{rem}

\begin{lemma}\label{lem:cost_of_component_neg}
    Let $C'$ be a repeated component of $C$ such that $\tau(W) < 0$ with $W:= \pi(C')$.
    Then we either have
    \begin{enumerate}[label = (\alph*)]
        \item $c(W) \geq -\sum_{a \in A} |c(a)|u(a)$ if $W$ lies in $\dynN_{\phi(f_i)}(s)$ or in $\dynN_{\phi(f_i)}(t)$ or
        \item  $$c(W) \geq \left( - 2\sum_{a \in A}\tau(a) -2  \right) \cdot \sum_{a \in A}|c(a)|u(a).$$
        with $c(W(t)) \geq -\sum_{a\in A}|c(a)|u(a) $, 
        \begin{align*}
        c(\overline{W}) \geq \left( - 2\sum_{a \in A}\tau(a) \right) \cdot \sum_{a \in A}|c(a)|u(a),
        \end{align*}
        and $c(W(s)) \geq -\sum_{a\in A}|c(a)|u(a) $ if $W$ is of the form (2) or (3) from Lemma~\ref{lem:structure_of_connected_component_neg}.
    \end{enumerate}

\end{lemma}
\begin{proof}
    Depending on the structure of $W$ shown in Lemma~\ref{lem:structure_of_connected_component_neg} we are able to derive a lower bound on the cost of $W$.
    If $W$ lies in $\dynN_{\phi(f_i)}(s)$ or $\dynN_{\phi(f_i)}(t)$, the statement of the lemma follows with Lemma~\ref{lem:walk_in_Ns_or_Nt_cost}.

    By Remark~\ref{rem} we suppose that $W$ is of the form (2). 
    In this case $W = W(t) \oplus \overline{W} \oplus{W}(s)$ with $W(t) \subseteq \dynN_{\phi(f_i)}(t)$, $W(s) \subseteq \dynN_{\phi(f_i)}(s)$, $\overline{W} \subseteq \overline{\dynN}_{\phi(f_i)}$ with  $h_{\phi(f_i)}(W(t)) < \sum_{a \in A}\tau(a)$ and $h_{\phi(f_i)}(W(t) \oplus \overline{W}) < 2\sum_{a \in A}\tau(a)$.
    %
    %
    %
    %
    Thus, in the worst case $\overline{W}$ runs upwards a height of at most $2\sum_{a \in A} \tau(a)-1 $ by repeatedly traversing through cycles $\overline{B}$ with $\tau(\overline{B}) = 1$ and $c(\overline{B}) \geq -\sum_{a \in A}|c(a)|u(a)$.
    This, together with $c(\overline{P}) \geq -\sum_{a \in A} |c(a)|u(a)$, gives
    \begin{align*}
        c(\overline{W}) \geq  c(\overline{P}) + \left(-2\sum_{a \in A} \tau(a) +1\right)\cdot \sum_{a \in A} |c(a)|u(a) = \left( - 2\sum_{a \in A}\tau(a) \right) \cdot \sum_{a \in A}|c(a)|u(a).
    \end{align*}
    %
    %
    The cost of the walks $W(s)$ and $W(t)$ are lower bounded by $-\sum_{a \in A}|c(a)|u(a)$ by Lemma~\ref{lem:walk_in_Ns_or_Nt_cost}.
    Hence,
    $$c(W)  = c(W(t) \oplus \overline{W} \oplus W(t)) >  \left( - 2\sum_{a \in A}\tau(a) -2  \right) \cdot \sum_{a \in A}|c(a)|u(a).$$
\end{proof}
\paragraph{Components with Strictly Positive Transit Time.}
\begin{lemma}\label{lem:no_cycle_crossing_cuts_pos}
    Let $C'$ be a repeated component of $C$ such that $\tau(W) > 0$ with $W:= \pi(C')$.
    Then $W$ lies completely in $\dynN_{\phi(f_i)}(s)$ or $\dynN_{\phi(f_i)}(t)$.
    More precisely, if $C$ contains $\hat{s}$, then $W$ lies completely in $\dynN_{\phi(f_i)}(s)$.
    Similarly, if $C$ contains $\hat{t}$, then $W$ lies completely in $\dynN_{\phi(f_i)}(t)$.
    Additionally, we have $c(C') >  -\sum_{a \in A}|c(a)|u(a) $.
\end{lemma}
\begin{proof}
    Define $C'' := C\setminus (C \cap \TENS{\theta_1^i}{\theta_2^i}_{f_i})$.
    Since $\phi(f_i)$ is a maximum $s$-$t$ flow by Lemma~\ref{lem:maxflow_induces_max_flow}, it is immediate that $W\cap \cut(\dynN) = \emptyset$.
    This directly implies that $C' \cap \lift(\cut(\dynN)) = \emptyset$.

    By Remark~\ref{rem} $C'$ is in $\TENS{1}{\theta_2^i}_{f_i}$.
    In this case $C'$ is a path that starts in layer $\theta_1^i$ and ends in some layer $\mu$ with $\theta_1^i < \mu \leq \theta_2^i$.
    In particular, $W$ ends with $s$ or $t$.
    If $W$ ends with $t$, then $W \subseteq \dynN_{\phi(f_i)}(t)$ as $W \cap \cut(\dynN) = 0$.

    Assume $W$ ends with $s$. 
    For the purpose of deriving a contradiction suppose that $C'$ contains a backward arc from $\lift(\cut(\dynN))$.
    Thus, $W$ starts in $\overline{\dynN}_{\phi(f_i)}$ or in $\dynN_{\phi(f_i)}(t)$.
    In either case, as $C$ contains $\hat{s}$, we have $C'' \cap \lift(\cut(\dynN)) \neq \emptyset$ which is a contradiction by Lemma~\ref{lem:augmenting_path}.
    Hence, $C'$ does not contain a backward arc from $\lift(\cut(\dynN))$ and $W$ lies completely in $\dynN_{\phi(f_i)}(s)$.

    Finally, consider the case that $C'$ intersects $\TENS{1}{\theta_1^i}_{f_i}$, $\TENS{\theta_1^i}{\theta_2^i}_{f_i}$ and $\TENS{\theta_2^i}{\theta}_{f_i}$.
    To prove this case we have to consider several subcases.
    First we suppose that $C$ does \emph{not} contain $\hat{s}$ or $\hat{t}$.
    In this case $C$ also has a repeated component $C_2$ that runs from layer $\theta_2^i$ down to layer $\theta_1^i$.
    By Theorem~\ref{lem:structure_of_connected_component_neg} $W_2 = \pi(C_2)$ decomposes into $W_2 = W_2(t) \oplus \overline{W}_2 \oplus W_2(s)$ with $W_2(t) \subseteq \dynN_{\phi(f_i)}(t)$, $W_2(s) \subseteq \dynN_{\phi(f_i)}(s)$ and $\overline{W}_2 \subseteq \overline{\dynN}_{\phi(f_i)}$ such that either $h_{\phi(f_i)}(W_2(s))>\sum_{a \in A}\tau(a)$ or  $h_{\phi(f_i)}(W_2(t))>\sum_{a \in A}\tau(a)$.
    First suppose that $h_{\phi(f_i)}(W_2(s))>\sum_{a \in A}\tau(a)$.
    In this case Lemma~\ref{lem:augmenting_path} implies that $C \cap \TENS{1}{\theta_1^i}_{f_i}$ does not contain an arc from $\lift(\cut(\dynN))$.
    Since, $W_2$ ends in $\dynN_{\phi(f_i)}(s)$, this implies that also no arc from $\lift(\cut(\dynN))$ is contained on the subpath of $C$ connecting $C_2$ and $C'$.
    Thus, $W$ also starts in $\dynN_{\phi(f_i)}(s)$ and hence lies completely in $\dynN_{\phi(f_i)}(s)$.

    Next suppose that $h_{\phi(f_i)}(W(t))>\sum_{a \in A}\tau(a)$ (and thus  $h_{\phi(f_i)}(W(s)) < \sum_{a \in A}\tau(a)$).
    By Lemma~\ref{lem:augmenting_path} this yields that $C \cap \TENS{\theta_2^i}{\theta}_{f_i}$ does not contain an arc from $\lift(\cut(\dynN))$.
    Consider the repeated component $C_3$ running from layer $\theta_2^i$ to layer $\theta_1^i$ that is traversed by $C$ after $C'$.
    Let $W_3 := \pi(C_3)$.
    Suppose that $W$ ends in $\dynN_{\phi(f_i)}(s)$ or $\overline{\dynN}_{\phi(f_i)}$.
    If $W$ ends in $\dynN_{\phi(f_i)}(s)$, then as $C \cap \TENS{\theta_2^i}{\theta}_{f_i}$ does not contain an arc from $\lift(\cut(\dynN))$ the repeated component $W_3$ lies completely in $\dynN_{\phi(f_i)}(s)$.
    If $W$ ends in $\overline{\dynN}_{\phi(f_i)}$, then $W_3$ ends with $W_3(s)$ such that $h_{\phi(f_i)}(W_3(s))\geq \sum_{a\in A}\tau(a)$. 
    In both cases Lemma~\ref{lem:augmenting_path} implies that $C \cap \TENS{1}{\theta_1^i}_{f_i}$ does not contain an arc from $\lift(\cut(\dynN))$.
    Since $W_3$ ends in $\dynN_{\phi(f_i)}(s)$ this overall yields that all repeated components of $C$ lie completely in $\dynN_{\phi(f_i)}(s)$, a contradiction.
    Thus, $W$ lies completely in $\dynN_{\phi(f_i)}(t)$.

    Next, we consider the case that $C$ contains $\hat{t}$.
    Lemma~\ref{lem:augmenting_path} implies that $C \cap \TENS{\theta_2^i}{\theta}_{f_i}$ does not contain an arc from $\lift(\cut(\dynN))$.
    We have to consider several subcases.

    First suppose that $C$ contains a repeated component $C_3$ running from layer $\theta_2^i$ to layer $\theta_1^i$ and suppose that $C_3$ is the repeated component of this form that is traversed by $C$ directly after $C'$.
    Let $W_3 := \pi(C_3)$. 
    Lemma~\ref{lem:augmenting_path} implies that $W_3$ does not contain a repeated component $W_3(s)$ with $W_3(s) \subseteq \dynN_{\phi(f_i)(t)}$ with $h_{\phi(f_i)}(W_3(s))\geq \sum_{a \in A}\tau(a)$ as $\hat{t} \in C$.
    Thus $W_3$ starts in $\dynN_{\phi(f_i)}(t)$.
    As $C \cap \TENS{\theta_2^i}{\theta}_{f_i}$ does not contain an arc from $\lift(\cut(\dynN))$ this implies that $W$ ends in $\dynN_{\phi(f_i)}(s)$ and thus lies completely in $\dynN_{\phi(f_i)}(t)$ because $W \cap \cut(\dynN) = \emptyset$.

    If $C$ does not contain any repeated component running from layer $\theta_2^i$ to layer $\theta_1^i$ that is traversed \emph{after} $C'$, then as $C \cap \TENS{\theta_2^i}{\theta}_{f_i}$ does not contain an arc from $\lift(\cut(\dynN))$ and $\hat{t} \in C$, we again get that $W$ ends in $\dynN_{\phi(f_i)}(t)$ and hence lies completely $\dynN_{\phi(f_i)}(t)$.
    The case that $\hat{s} \in C$ can be shown by symmetric arguments.
\end{proof}
\paragraph{Components with Zero Transit Time.}
\begin{lemma}\label{lem:no_cycle_crossing_cuts_0}
    Let $C'$ be a repeated component of $C$ such that $\tau(W) = 0$ with $W:= \pi(C')$.
    If $h^\theta_{f_i}(C') \geq  2\left(\sum_{a \in A} \tau(a)\right)$, then $C'$ does not contain any forward or backward arc from $\lift(\cut(\dynN))$ and $W$ lies in $\dynN_{\phi(f_i)}(s)$ or $\dynN_{\phi(f_i)}(t)$.
\end{lemma}
\begin{proof}
    By Remark~\ref{rem}, $C'$ is in $\TENS{1}{\theta_2^i}_{f_i}$.
    Since $h^\theta_{f_i}(C') > 2\left(\sum_{a \in A} \tau(a)\right)$ and because there are no cycles with negative transit time in $\overline{\dynN}_{\phi(f_i)}$ by Lemma~\ref{Cor:no_neg_transit_time_cycle} there is a $v \in W$ with $v \in \dynN_{\phi(f_i)}(s)$ or $v \in \dynN_{\phi(f_i)}(t)$ with $h_{\phi(f_i)}(v,W) \geq \sum_{a \in A}\tau(a)$.
    First assume that $v \in \dynN_{\phi(f_i)}(s)$.
    Aiming for a contradiction, suppose that $C'$ contains a backward arc from $\lift(\cut(\dynN))$.
    We thus get that $W$ ends in $\dynN_{\phi(f_i)}(s)$ and starts either in $\overline{\dynN}_{\phi(f_i)}$ or $\dynN_{\phi(f_i)}(t)$.
    Independent of the specific form of the cycle $C$ this implies that $C'' = C\setminus (C \cap \TENS{\theta_1^i}{\theta_2^i}_{f_i})$ contains an arc $a^{\theta'} \in \lift(\cut(\dynN))$ with $\theta' \in [1,\theta_1^i)$ contradicting Lemma~\ref{lem:augmenting_path}.
    Thus, $C'$ does not contain a backward arc from $\lift(\cut(\dynN))$ and hence $W$ lies completely in $\dynN_{\phi(f_i)}(s)$.

    If $v \in \dynN_{\phi(f_i)}(t)$ then, since $W$ does not contain an arc in $\cut(\dynN)$, the whole path $W$ lies in $\dynN_{\phi(f_i)}(t)$ and thus $C'$ does not contain any backward arc from $\dynN_{\phi(f_j)}(t)$.
    %
\end{proof}
\begin{lemma}\label{lem:cost_of_component_0}
    Let $C'$ be a repeated component of $C$ such that $\tau(W) = 0$ with $W:= \pi(C')$.
    If $h^\theta_{f_i}(C') > 2\left(\sum_{a \in A} \tau(a)\right)\cdot \left( \sum_{a \in A}|c(a)|u(a)+1\right) $, then $c(C') > 0$.
\end{lemma}
\begin{proof}
    By Remark~\ref{rem}, $C$ is in $\TENS{1}{\theta_2^i}_{f_i}$.
    By our assumption on the height of $W$, we can apply Lemma~\ref{lem:no_cycle_crossing_cuts_0} and deduce that $W$ either lies completely in $\dynN_{\phi(f_i)}(s)$ or  $\dynN_{\phi(f_i)}(t)$.
    The walk $W$ can be decomposed into a path $P$ with $h_{\phi(f_i)}(P) \leq \sum_{a \in A} \tau(a)$ and $c(P) \geq -\sum_{a\in A}|c(a)|u(a)$, and cycles $B$ that are $s$-reachable or $t$-reachable by Lemma~\ref{lem:cycles_are_reachable} and hence fulfill $c(B)\geq 0$ by construction.
    Since $h_{\phi(f_i)}(W) > 2\left(\sum_{a \in A} \tau(a)\right)\cdot \left( \sum_{a \in A}|c(a)|u(a)+1\right)$ and $\tau(W) = 0$ by assumption, there is at least one cycle $B$ in the decomposition of $W$ with $\tau(B) \leq -1$ and thus $c(B) \geq 1$ by Lemma~\ref{lem:properties} part (b) that is traversed at least once. 
    As $h_{\phi(f_i)}(P) \leq \sum_{a \in A} \tau(a)$, in the walk $W$ at least a height of $h_{\phi(f_i)}(W)-\sum_{a\in A}\tau(a)$ has to be traversed by traveling along $B$ multiple times.
    Traveling along $B$ once changes the height by at most $\sum_{a \in A} \tau(a)$ while contributing a cost of a least $1$.
    This yields $C(W) > 0$.
    %
    %

    %
\end{proof}
We conclude this subsection by deriving a lower bound on the cost of $C$.
\begin{lemma}\label{lem:aux_height}
    The repeated components of $C \cap \TENS{\theta_1^{i}}{\theta_2^{i}}$ can be decomposed into at most $|V|+1$ repeated components  $C^1_+,\ldots,C^{k_+}_+$ with strictly positive transit time, at most $|V|+1$ repeated components $C^1_-,\ldots,C^{k_-}_-$ with strictly negative transit time, and at most $2|V|$ repeated components $C_0^1,\ldots,C_0^{k_0}$ with zero transit time with
    \begin{align}\label{eq:bound_on_cost_overall}
        c(C) = \sum_{j = 1}^{k_+} c(C_+^j)+ \sum_{j = 1}^{k_-} c(C_-^j) + \sum_{j = 1}^{k_0} c(C_0^j) \geq  \left( (|V|+1)\left(- 2\sum_{a \in A}\tau(a) -3\right)-2h  \right) \cdot \sum_{a \in A}|c(a)|u(a)
    \end{align}
    with $h = 2\left(\sum_{a\in A}\tau(a)\right) \cdot \left( \sum_{a \in A}|c(a)|u(a) + 1\right)$ defined as in~\ref{eq:def_h}.
\end{lemma}
\begin{proof}
    That the repeated components can be decomposed in such a way follows from Lemma~\ref{lem:num_of_connected_components}.
    Lemma~\ref{lem:cost_of_component_0} and Lemma~\ref{lem:num_of_connected_components} imply that in the worst case the repeated components with zero transit time of height at most $h$ have a cost of at most $-2h\sum_{a \in A}|c(a)|u(a)$.
    This worst case can only take place if $C$ intersects $\TENS{1}{\theta_1^i}_{f_i}$, $\TENS{\theta_1^i}{\theta_2^i}_{f_i}$ and $\TENS{\theta_2^i}{\theta}_{f_i}$.
    Also by Lemma~\ref{lem:cost_of_component_0}, the repeated components with zero transit time of height larger than $h$ contribute a strictly positive cost.  
    Thus, we have
    \begin{align}\label{eq:cost_zero}
        \sum_{i = 1}^{k_0} c(C_0^i) \geq -2h\sum_{a \in A}|c(a)|u(a).
    \end{align}
    By Lemma~\ref{lem:no_cycle_crossing_cuts_pos} and Lemma~\ref{lem:num_of_connected_components} we also have
    \begin{align}\label{eq:cost_pos}
       \sum_{i = 1}^{k_+} c(C_+^i) \geq -(|V|+1)\sum_{a \in A}|c(a)|u(a),
    \end{align}
    while Lemma~\ref{lem:cost_of_component_neg} and Lemma~\ref{lem:num_of_connected_components} yield,
    \begin{align}\label{eq:cost_neg}
        \sum_{i = 1}^{k_+} c(C_+^i) \geq (|V|+1)\left( - 2\sum_{a \in A}\tau(a) -2  \right) \cdot \sum_{a \in A}|c(a)|u(a).
    \end{align}
    Putting together~\eqref{eq:cost_zero},~\eqref{eq:cost_pos} and~\eqref{eq:cost_neg} yields the statement of the lemma.
\end{proof}
\subsection{Bounding the Height of Each Component}
In this subsection we derive lower bounds on the height of each repeated component of $C$ depending on its transit time.
If $C'$ is a repeated component with positive transit time that lies in $\TENS{1}{\theta_2^i}_{f_i}$ or in $\TENS{\theta_1^i}{\theta}_{f_i}$, then it follows directly from the properties of the compression procedure described in Appendix~\ref{SecA:Compression} in Lemma~\ref{lem:compression} that $C'$ has bounded height. 
\begin{lemma}\label{lem:height_of_component_pos}
    Let $C'$ be a repeated component of $C$ in $\TENS{1}{\theta_2^i}_{f_i}$ or in $\TENS{\theta_1^i}{\theta}_{f_i}$ such that $\tau(W) > 0$ with $W:= \pi(C')$ .
    Then $h^\theta_{f_i}(C') < J(h_i) $ with $J(h_i)$ defined as in~\eqref{eq:def_bound}.
\end{lemma}
We are now prepared to give a bound on the height of repeated components of $C$ with negative (or zero) transit time.
\begin{lemma}\label{lem:height_of_component_neg}
    Let $C'$ be a repeated component of $C$ such that $\tau(W) \leq 0$ with $W:= \pi(C')$.
    Then $h^\theta_{f_i}(C') \leq  J(h_i)$ with $J(h_i)$ defined as in~\eqref{eq:def_bound}.
\end{lemma}
\begin{proof}
    With the purpose of deriving a contradiction, assume that $h^\theta_{f_i}(C') >  J(h_i)$.
     Note that for $C'' = C\setminus ( C \cap \TENS{\theta_1^i}{\theta_2^i}_{f_i})$, we have
    \begin{align*}
        c(C'') \geq -2h_i\sum_{a \in A}|c(a)|u(a).
    \end{align*}
    Using this together with~\eqref{eq:bound_on_cost_overall} in Lemma~\ref{lem:aux_height}, we obtain
    \begin{align} \label{eq:bound_remaining}
        c(C) \geq c(C') + \left( (|V|+1)\left(- 2\sum_{a \in A}\tau(a) -3\right)-2h -2h_i \right) \cdot \sum_{a \in A}|c(a)|u(a).
    \end{align}
    If $\tau(W) = 0$, we can apply Lemma~\ref{lem:no_cycle_crossing_cuts_0} and deduce that $W$ lies in $\dynN_{\phi(f_i)}(s)$ or $\dynN_{\phi(f_i)}(t)$.
    If $\tau(W) < 0$, then we can derive a lower bound on the cost of $C'$ by exploiting the structure of $W$ according to Lemma~\ref{lem:structure_of_connected_component_neg}.
    We can handle the cases $\tau(W) = 0$ and $\tau(W)<0$ and $W$ is of the form (1) together.

    In either case $W$ lies in $\dynN_{\phi(f_i)}(s)$ or $\dynN_{\phi(f_i)}(t)$.
    Thus, $W$ decomposes into a path $P$ with $h_{\phi(f_i)}(P) \leq \sum_{a \in A}\tau(a)$ and cycles $B$ that are $s$-reachable or $t$-reachable by Lemma~\ref{lem:cycles_are_reachable} and hence fulfill $c(B)\geq 0$ by construction.
    Because $h_{\phi(f_i)}(W) > J(h_i)$ by assumption, there is a cycle $B$ in the decomposition of $W$ with $\tau(B)\leq -1$ and thus $c(B) \geq 1$ by Lemma~\ref{lem:properties}.
    As $h_{\phi(f_i)}(P) \leq \sum_{a \in A} \tau(a)$, in the walk $W$ at more than a height of $h_{\phi(f_i)}(W)-\sum_{a\in A}\tau(a)$ has to be traversed by traveling along $B$ multiple times.
    Traveling along a cycle like $B$ once changes the height by at most $\sum_{a \in A} \tau(a)$ while contributing a cost of a least $1$.
     This yields 
    \begin{align*}
        c(C') > \left( (|V|+1)\left(2\sum_{a \in A}\tau(a) +3 \right)+1+2h+2h_i  \right) \cdot \sum_{a \in A}|c(a)|u(a) 
    \end{align*}
    With~\eqref{eq:bound_remaining} we thus get $c(C) > 0$, a contradiction.

    Next, suppose that $\tau(W) < 0$ and $W$ is of the form (2) from Lemma~\ref{lem:structure_of_connected_component_neg}, i.e., $W = W(t) \oplus \overline{W} \oplus{W}(s)$ with $W(t) \subseteq \dynN_{\phi(f_i)}(t)$, $W(s) \subseteq \dynN_{\phi(f_i)}(s)$ and $\overline{W} \subseteq \overline{\dynN}_{\phi(f_i)}$.
    Lemma~\ref{lem:cost_of_component_neg} implies that
    $c(W(t)) \geq -\sum_{a\in A}|c(a)|u(a) $ and 
    \begin{align*}
        c(\overline{W}) \geq \left( - 2\sum_{a \in A}\tau(a) \right) \cdot \sum_{a \in A}|c(a)|u(a).
    \end{align*}
    The walk $W(s)$ decomposes into a path $P(s)$ with $c(P(s))\geq-\sum_{a\in A}|c(a)|u(a)$ and $h_{\phi(f_i)}(P(s)) \leq \sum_{a\in A} \tau(a)$ and cycles. 
    Because of $h_{\phi(f_i)}(W) > J(h_i)$ by assumption, there is at least one cycle $C(s)$ with $\tau(C(s)) \leq -1$ and $c(C(s)) \geq 1$ by Lemma~\ref{lem:properties}.
    With $h_{\phi(f_i)}(W(t) \oplus \overline{W} \oplus P(s))  < 3\sum_{a \in A} \tau(a)$ by Lemma~\ref{lem:structure_of_connected_component_neg} we note that the height changes more than
    \[
        h_{\phi(f_i)}(W) - 3 \sum_{a \in A}\tau(a) > \left(\sum_{a \in A}\tau(a)\right)\left( (|V|+1)\left(2\sum_{a \in A}\tau(a) +3 \right)+1+2h+2h_i  \right) \cdot \sum_{a \in A}|c(a)|u(a)
    \]
    by traveling along $C(s)$ multiple times. 
   Traveling along such a cycle changes the height by at most $\sum_{a \in A}\tau(a)$ and contributes at least a cost of $1$.
    Thus, we get 
   \begin{align*}
   c(W(s)) &> c(P(s)) + \left( (|V|+1)\left(2\sum_{a \in A}\tau(a) +3 \right)+1+2h+2h_i  \right) \cdot \sum_{a \in A}|c(a)|u(a) \\ &>  \left( (|V|+1)\left(2\sum_{a \in A}\tau(a) +3 \right)+2h+2h_i  \right) \cdot \sum_{a \in A}|c(a)|u(a)
   \end{align*}
   and hence $c(C)> 0$, a contradiction.
    The case that $W$ is of the form (3) from Lemma~\ref{lem:structure_of_connected_component_neg} can be handled analogously.
\end{proof} 

\subsection{Proof of Lemma~\ref{lem:Step2_repeated}}
Let $i+1$ be an iteration in which Step 2 is visited and suppose that $f_i$ is repeated during $[\theta_1^i,\theta_2^i]$ with $\theta_2^i-\theta_1^i > 2J(h_i)$ with $h_i = \max\{\theta_1^i,\theta-\theta_2^i\}$.
Let $C$ be a compressed cycle along which we augment in iteration $i+1$.
Note that we can now use Lemma~\ref{lem:height_of_component_pos} and Lemma~\ref{lem:height_of_component_neg} to deduce that $h_{f_i}^\theta(C) \leq J(h_i)$ and thus $f_{i+1}$ is repeated during  $[\theta_1^i+J(h_i),\theta_2^i - J(h_i)]$.
Applying this argument inductively with $f_{j_2}$ as the base case does not directly give a proof of Lemma~\ref{lem:Step2_repeated} as we do not have a bound on the number of iterations that Step 2 requires. 
If Step 2 requires a number of iterations in the order of the given time-horizon $\theta$, applying the argument above inductively does \emph{not} yield that the flow over time computed by Step 2 is repeated.
This is where the extra augmentations in lines~\ref{alg:aug1} and~\ref{alg:aug2} of our construction come into play.
These augmentations ensure that after every iteration $i+1$ that visits Step 2 the resulting flow over time is repeated during $[\theta_1^{j_2}+J(\theta_1^{j_2}),\theta_2^{j_2}+J(\theta_1^{j_2})]$.

There is one additional auxiliary lemma that we need in order to prove Lemma~\ref{lem:Step2_repeated}. 
 To this end, consider the Eulerian Graphs $H_{i+1}$ and $H'_{i+1}$ in iteration $i+1$ in which Step 2 is visited.
Suppose that the Eulerian Graph $H_{i+1}$ decomposes into disjoint cycles $K_{i+1,1},\ldots,K_{i+1,k_{i+1}}$.

\begin{lemma}\label{lem:cost_of_Eulerian_Graph}
    Let $i+1$ be an iteration in which Step 2 is visited and suppose that $f_i$ is repeated during $[\theta_1^i,\theta_2^i]$ with $\theta_2^i-\theta_1^i > 2J(h_i)$ with $h_i = \max\{\theta_1^i,\theta-\theta_2^i\}$.
    We  have $c(K_{i+1,\mu}) < 0$ and $K_{i+1,\mu} \cap \TENS{1}{\theta_1^{j_2}}_{f_{j_2}} \neq \emptyset$ or $K_{i+1,\mu} \cap \TENS{\theta_2^{j_2}}{\theta}_{f_{j_2}} \neq \emptyset$  for all $\mu \in \{1,\ldots, k_{i+1}\}$.
\end{lemma}
\begin{proof}
    We give a proof by induction.
    Iteration $j_2+1$ is the first iteration in which Step 2 is visited.
    In this iteration, we augment along the compressed cycle $C_{j_2+1}$ with $c(C_{j_2+1})<0$.
    Thus, $C$ does not lie in $\TENS{\theta_{1}^{j_2}}{\theta_2^{j_2}}_{j_{j_2}}$ by Lemma~\ref{lem1}.
    %
    %
        %
    Note that we have $k_{j_2+1} = 1$.
    Thus, $H_{j_2+1}$ only consists of the cycle $C_{j_2+1}$ and hence the assertion of the lemma is true for $i = j_2+1$.

    Next, let $i+1>j_2+1$.
    In this iteration we augment along the compressed cycle $C_{i+1}$ with $c(C_{i+1}) < 0$.
    If $C_{i+1}$ is not induced by any of the cycles along which we augmented in previous iterations where Step 2 was visited, then $C_{i+1}$ is a cycle in $\TEN{\theta}_{f_{j_2}}$ and the Eulerian graph $H_{i+1}$ decomposes into $H_{i}$ and $C_{i+1}$.
    Thus, by induction, $H_{i+1}$ decomposes into cycles with strictly negative cost.
    Since $c(C^{i+1}) < 0$, the cycle $C^{i+1}$ does not lie in $\TENS{\theta_1^{j_2}}{\theta_2^{j_2}}_{f_{j_2}}$ by Lemma~\ref{lem1}.
    Thus, $C_{i+1}$ intersects $\TENS{1}{\theta_1^{j_2}}_{f_{j_2}}$ or  $\TENS{\theta_2^{j_2}}{\theta}_{f_{j_2}}$.
    Overall, the assertion of the lemma is true in this case.

    Next, suppose that $C^{i+1}$ is induced by cycles along which we augmented in previous iterations where Step 2 was visited.
    We consider the Eulerian graph $H_{i+1}'$ which is induced by $H_{i}$ and $C_{i+1}$ by removing opposite arcs.
    Arcs that occur in both $H_{i}$ and $C_{i+1}$ are counted multiple times.
    The Eulerian graph $H_{i+1}'$  decomposes into disjoint cycles $K'_{i+1,1},\ldots, K'_{i+1,\ell_{i+1}}$.
    We prove that $c(K'_{i+1,\mu}) \leq 0$ for all $\mu \in \{1,\ldots,\ell_{i+1}\}$.
    At first we consider the case $\ell_{i+1} = 1$. 
    In this case, we have, using that $c(H_{i}) < 0$ by induction,
    \begin{align*}
        c(H'_{i+1}) = c(K'_{i+1,1}) = c(H_{i}) + c(C_{i+1}) < 0.
    \end{align*}
    Next, suppose that $\ell_{i+1} >1$ and that $c(K'_{i+1,1}) > 0$.
    The cycle $K'_{i+1,1}$ is induced by $C_{i+1}$ as otherwise $K'_{i+1,1}$ would be contained in a decomposition of $H_{i}$
    and hence fulfill $c(K'_{i+1,1})<0$ by induction.
    Denote by $K''_{i+1,1}$ the cycle $K'_{i+1,1}$ in backward direction.
    Consider the cycle $B$ induced by $K''_{i+1,1}$ and $C_{i+1}$ by removing opposite arcs.
    The cycle $B$ lies in $\TENS{1}{\theta}_{f_{i}}$ and $c(B) = c(C_{i+1}) + c(K''_{i+1,1}) < c(C_{i+1})$, contradicting the minimality of the cost of $C_{i+1}$.
    Thus, we have $c(K'_{i+1,\mu}) \leq 0$ for all $\mu \in \{1,\ldots,\ell_{i+1}\}$.
    However, we augment along all cycles $K'_{i+1,\mu}$ with $c(K'_{i+1,\mu}) = 0$ in lines~\ref{alg:aug1} and~\ref{alg:aug2} of our construction.
    Thus, $H_{i+1}$ decomposes into cycles in $\TEN{\theta}_{f_{j_2}}$ with strictly negative cost.
    Consider $K_{i+1,\mu}$ with $\mu \in \{1,\ldots, k_{i+1}\}$.
    The cycle $K_{i+1,\mu}$ lies in $\TEN{\theta}_{f_{j_2}}$ and $c(K_{i+1,\mu})<0$ by what we argued before.
    This again implies that $K_{i+1,\mu}$ intersects $\TENS{1}{\theta_1^{j_2}}$ or  $\TENS{\theta_2^{j_2}}{\theta}$ by Lemma~\ref{lem1}.
\end{proof}
We are now prepared to give a proof for Lemma~\ref{lem:Step2_repeated}.
\begin{proof}[Proof of Lemma~\ref{lem:Step2_repeated}].
    In order to show that $f_{i+1}$ is repeated during $[\theta_1^{j_2}+J(\theta_1^{j_2}),\theta_2^{j_2}-J(\theta_1^{j_2})]$, we prove the following statement:
    Let $K_{i+1,1},\ldots,K_{i+1,k_{i_1}}$ be the disjoint cycles that the Eulerian Graph $H_{i+1}$ decomposes into.
    Then for all $\mu \in \{1,\ldots, {k_{i+1}}\}$ the cycle $K_{i+1,\mu}$ either lies in $\TENS{1}{\theta_1^{j_2}+J(\theta_1^{j_2})}_{f_{j_2}}$ or in $\TENS{\theta_2^{j_2}-J(\theta_1^{j_2})}{\theta}_{f_{j_2}}$.

    We prove this statement by induction.
    First consider the iteration $j_2+1$ of Construction~\ref{alg}.
    Let $C_{j_2+1}$ be the negative cost cycle along which we augment in this iteration.
    Let $C'_{j_2+1} := C_{j_2+1} \cap \TENS{\theta_1^{j_2}}{\theta_2^{j_2}}_{f_j}$.
    By Lemma~\ref{lem1} the flow over time $f_{j_2}$ is repeated during $[\theta_1^{j_2},\theta_2^{j_2}]$.
    Thus, we can apply Lemma~\ref{lem:height_of_component_pos}, or Lemma~\ref{lem:height_of_component_neg} and deduce that $h^\theta_{f_{j_2}}(C'_{j_2+1}) \leq J(\theta_1^{j_2})$, i.e., $C_{j_2+1}$ lies in $\TENS{1}{\theta_1^j+J(\theta_1^{j_2})}_{f_{j_2}}$ or in $\TENS{\theta_2^j-J(\theta_1^{j_2})}{\theta}_{f_{j_2}}$.
    Furthermore, we have that $k_{j_2+1} = 1$.
    Hence, the Eulerian graph $H_{j_2+1}$ just consists of the cycle $C_{j_2+1}$ and thus the statement above is true in this case.

    Next, let $i+1>j_2+1$. 
    By induction, the flow over time $f_i$ is repeated during $[\theta_1^i,\theta_2^i] \supseteq [\theta_1^{j_2}+J(\theta_1^{j_2}),\theta_2^{j_2}-J(\theta_1^{j_2})]$.
     Let $C_{i+1}$ be the negative cost cycle along which we augment in iteration $i+1$ and define $C_{i+1}' := C_{i+1} \cap \TENS{\theta_1^{i}}{\theta_2^{i}}_{f_{i}}$.
    If $C_{i+1} \cap H_{i} = \emptyset$, then $C_{i+1}$ is a cycle in $\TEN{\theta}_{f_{j_2}}$ with $c(C_{i+1})<0$ which is either contained in $\TENS{1}{\theta_1^{j_2}+J(\theta_1^{j_2})}_{f_j}$ or in $\TENS{\theta_2^{j_2}-J(\theta_1^{j_2})}{\theta}_{f_{j_2}}$ by Lemma~\ref{lem:height_of_component_pos}, or Lemma~\ref{lem:height_of_component_neg} and the Eulerian graph $H_{i+1}$ decomposes into $K_{i,1},\ldots, K_{i,k_{i}}$ and $C_{i+1}$.
    By induction each of these cycles either lies in $\TENS{1}{\theta_1^{j_2}+J(\theta_1^{j_2})}_{f_{j_2}}$ or in $\TENS{\theta_2^{j_2}-J(\theta_1^{j_2})}{\theta}_{f_{j_2}}$.

    Next suppose that $H_{i} \cap C_{i+1} \neq \emptyset$.
    By induction, all of the cycles $K_{i,1},\ldots, K_{i,k_{i}}$ either lie in $\TENS{1}{\theta_1^{j_2}+J(\theta_1^{j_2})}_{f_{j_2}}$ or in $\TENS{\theta_2^{j_2}-J(\theta_1^{j_2})}{\theta}_{f_{j_2}}$.
    Recall that by construction $C_{i+1}$ either lies in $\TENS{1}{\theta_2^{j_2}-J(\theta_1^{j_2})}_{f_{i}}$ or in $\TENS{\theta_1^{j_2}+J(\theta_1^{j_2})}{\theta}_{f_{i}}$.
    We only consider the first case. 
    The argument for the latter case is symmetric.
    We can apply Lemma~\ref{lem:height_of_component_pos}, or Lemma~\ref{lem:height_of_component_neg} and deduce that $C_{i+1}$ lies in $\TENS{1}{\theta_1^{j_2}+J(\theta_1^{j_2})+J(h_i)}_{f_{i}}$.
    Note that $f_i$ is repeated during $[\theta^1_{j_2}+J(\theta_1^{j_2}),\theta^1_{j_2}+J(\theta_1^{j_2})]$ thus we have that $h_i = \theta^1_{j_2}+J(\theta_1^{j_2})$ and hence $h_i$ does not depend on $i$.
    As the time-horizon $\theta$ is large, this yields that $C_{i+1}$ can only interact with a cycle $K_{i,\mu}$ that lies in $\TENS{1}{\theta^{j_2}_{1}+J(\theta_1^{j_2})}_{f_{j_2}}$.
    Thus, the Eulerian graph $H_{i+1}$ decomposes into cycles $K_{i+1,1},\ldots, K_{i+1,k_{i+1}}$ that either lie in $\TENS{1}{\theta_2^{j_2}}_{f_{j_2}}$ or in $\TENS{\theta_1^{j_2}}{\theta}_{f_{j_2}}$.
    Note that $c(K_\mu^{i+1})<0$ for all $\mu \in \{1,\ldots,\ell_{i+1}\}$ by Lemma~\ref{lem:cost_of_Eulerian_Graph}.
    From Lemma~\ref{lem:height_of_component_pos}, or Lemma~\ref{lem:height_of_component_neg} we can now deduce that for $W_{\mu} = K^{i+1}_{\mu} \cap \TENS{\theta_1^{j_2}}{\theta_2^{j_2}}_{f_{j_2}}$ we have $h_{\phi(f_{j_2})}(W_{\mu}) \leq J(\theta_1^{j_2})$. 
    Overall, we have thus deduced that $f_i$ is repeated during $[\theta_1^{j_2}+J(\theta_1^{j_2}),\theta_2^{j_2}-J(\theta_1^{j_2})]$.
\end{proof}
\subsection{Final Steps of the Analysis}
Let $j_3$ be the last iteration in which Step 2 of the construction is visited.
From Lemma~\ref{lem:Step2_repeated} it follows that $f_{j_3}$ is a maximum flow over time that is repeated during $[\theta^{j_3}_1,\theta_2^{j_3}] = [\theta_1^{j_2} + J(\theta_1^{j_2}),\theta-\theta_2^{j_2}-J(\theta_1^{j_2})]$.

It remains to be shown that $f_{j_3}$ is of minimal cost.
By Lemma~\ref{lem1} we have that $\TENS{\theta_1^{j_3}}{\theta_2^{j_3}}_{f_{j_3}}$ does not contain a cycle of negative cost.
By construction, we also eliminated all negative cost cycles $C$ in $\TENS{1}{\theta_2^{j_3}}_{f_{j_3}}$ and  $\TENS{\theta_1^{j_3}}{\theta}_{f_{j_3}}$ in Step 2 of our construction.
To finish the proof of Theorem~\ref{thm:main}, it remains to be checked that $\TEN{\theta}_{f_{j_3}}$ does not contain negative cost cycles crossing all three ``areas'' of the time-expanded network, i.e., negative cost cycles that intersect $\TENS{\theta_1^{j_3}}{\theta_2^{j_3}}_{f_{j_3}}$ \emph{and} $\TENS{1}{\theta_2^{j_3}}_{f_{j_3}}$ \emph{and} $\TENS{\theta_1^{j_3}}{\theta}_{f_{j_3}}$. 
Such negative cost cycles are depicted in Figure~\ref{fig:cycles2}.
\begin{lemma}\label{lem:big_neg_component}
    The network $\TEN{\theta}_{f_{j_3}}$ does not contain a cycle $C$ with $c(C) < 0$ such that $C \cap \TENS{\theta_1^{j_3}}{\theta_2^{j_3}}$ has a repeated component $C'$  that is a path from layer $\theta_2^{j_3}$ to layer $\theta_1^{j_3}$.
\end{lemma}
\begin{proof}
    Define $W := \pi(C')$.
    We have 
    $$h_{f_{j_3}}(W) = \theta_2^{j_3} - \theta_1^{j_3} = \theta - 2(\theta_1^{j_2}+J(\theta_1^{j_2}))>2J(\theta_1^{j_2} + h(\theta^{j_2}_1)) = 2J(h_{j_3}),$$
    by our assumptions on $\theta$.
    This contradicts Lemma~\ref{lem:height_of_component_neg}.
   
\end{proof}
\begin{lemma}\label{lem:big_pos_component}
    The network $\TEN{\theta}_{f_{j_3}}$ does not contain a cycle $C$ with $c(C) < 0$ such that $C \cap \TENS{\theta_1^{j_3}}{\theta_2^{j_3}}$ has a repeated component $C'$  that is a path from layer $\theta_1^{j_3}$ to layer $\theta_2^{j_3}$.
\end{lemma}

\begin{proof}
    Aiming for a contradiction, let $C$ be a negative cost cycle in $\TEN{\theta}_{f_{j_3}}$ that contains such a repeated component $C'$.

    Define $C_-$ and $C_+$ such that $C = C_- \oplus C' \oplus C_+$.
    By Lemma~\ref{lem:big_neg_component}, $C$ contains $\hat{s}$ or $\hat{t}$.
    Suppose that $C$ contains $\hat{s}$.
    The other case can be shown by a symmetric argument.
    Lemma~\ref{lem:no_cycle_crossing_cuts_pos} implies that $W$ lies completely in $\dynN_{\phi(f_i)}(s)$.
    In particular, all cycles from the decomposition of $W$ are $s$-reachable by Lemma~\ref{lem:properties} and hence have positive cost by construction.
    Since $h^\theta_{f_i}(C') = \theta_2^{j_3}-\theta_1^{j_3}>2h(j_3)$, there is at least one cycle $B$ in the decomposition of $W$ with $\tau(B) > 0$ and $c(B) \leq 0$.
    If $c(B)<0$, then we can derive a contradiction analogously to the proof of Lemma~\ref{lem:height_of_component_neg}.
    Note that neither $C_-$ nor $C_+$ can be ``shifted'' into the repeated area of $\TEN{\theta}_f$.
    Suppose we can shift $C_+$ into the repeated area of the time-expanded network.
    After potentially removing cycles with strictly positive transit time from $C'$ we can then create a cycle $\tilde{C}$ of the form 
     $c(\tilde{C}) \leq c(C) < 0$ that lies in $\TENS{1}{\theta_2^{j_3}}_{f_{j_3}}$.
    However, such a cycle does not exist after Step 2 of the construction.

    Thus, $\pi(C_-)$ contains an arc $a_-$ that is ``blocked'' in $\dynN_{\phi(f)}$.
    Similarly, there is an arc $a_+$ in $\pi(C_+)$ that is also ``blocked'' in $\dynN_{\phi(f_i)}$.
    Here ``blocking'' $a_+$ means to block an $\hat{s}$-$\hat{t}$ path along which $f_{j_3}$ can send flow via $\cev{a}_+$.
    We distinguish whether $a_+$ and $a_-$ are forward or backward arcs in $\dynN$.
    Without loss of generality suppose that $a_+$ is a backward arc and $a_-$ is a forward arc.
    Then, in $\phi(f_{j_3})$ flow is sent along a path $P_+$ that  blocks $a_+$.
    Suppose that $P_+$ crosses $\cut(\dynN)$ at some arc $a_{\cut}$.
    Eventually, $f_{j_3}$ stops to send flow along liftings of $P_+$ so that in $\TENS{\theta_2}{\theta}$ liftings of the arc $a_+$ can be utilized.
    Note that $C_+$ ends in $\hat{s}$, i.e., $C_+$ cancels out flow, that is send out of $\hat{s}$ via some walk $P'$ in $\dynN$ towards some arc $a_{\cut}' \in \cut(\dynN)$.
    We argue that $a_{\cut} \neq a'_{\cut}$.
    Aiming for a contradiction, suppose that $a_{\cut}' = a_{\cut}$.
    See Figure~\ref{fig:basecase1} for an illustration of the situation in $\dynN$.   
    \begin{figure}
        \centering
        \begin{subfigure}{0.4\textwidth}
            \includegraphics{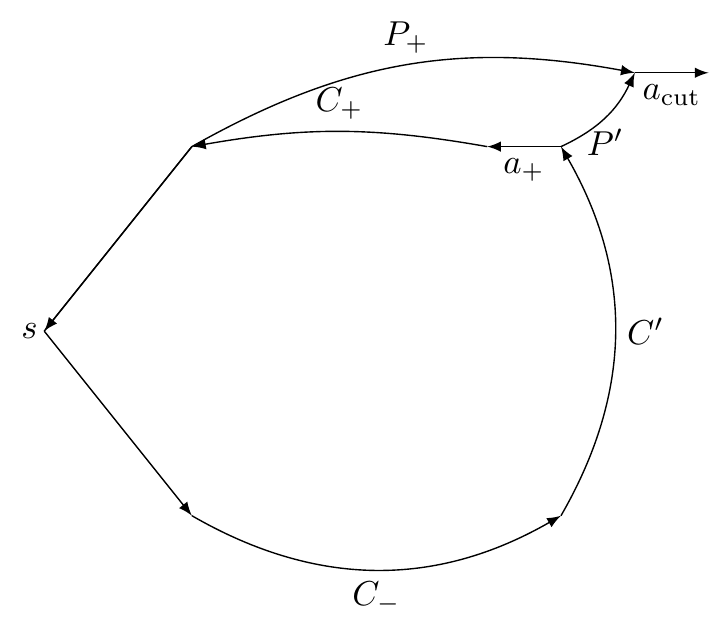} 
            \caption{The situation if we assume $a_{\cut} = a_{\cut}'$}
            \label{fig:basecase1}    
        \end{subfigure}
        \hfill
        \begin{subfigure}{0.4\textwidth}
        \centering
        \includegraphics{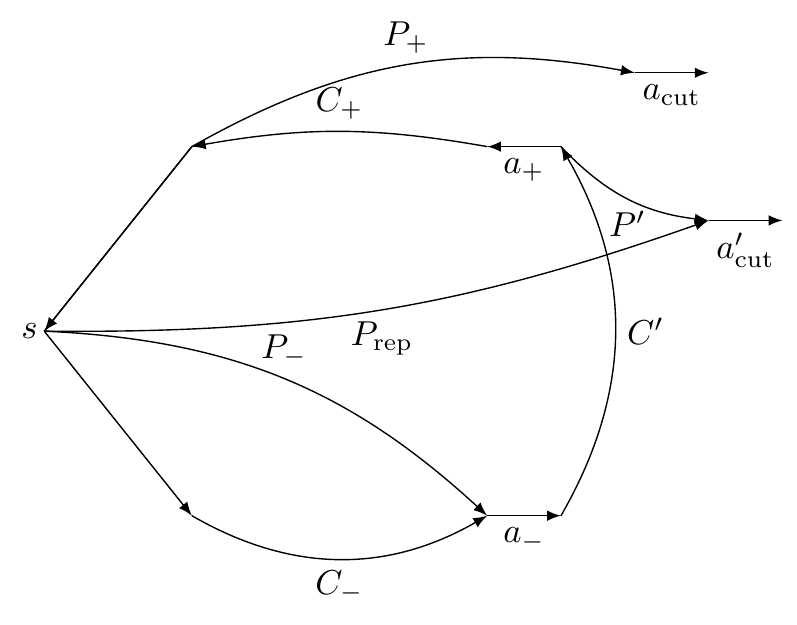} 
        \caption{The situation if we assume $a_{\cut} \neq a_{\cut}'$}
        \label{fig:basecase2}              
    \end{subfigure}  
    \caption{The situations under the assumption $a_{\cut} = a_{\cut}'$ or $a_{\cut} \neq a_{\cut}'$ }
    \end{figure}           
    In the repeated part of the flow over time $f_{j_3}$, flow is sent along $a_{\cut}$ via $P_+$ which blocks the walk $P'$ along which flow is sent along $a_{\cut}$ in the non-repeated area.
    Overall, this yields that $\cev{P}_+$ and $P'$ induce a cycle $B_1$ in $\dynN_{\phi(f)}$ with $c(B_1) = c(W)-c(P_+) \geq 0$ as $B_1$ is $s$-reachable by construction.
    Similarly, we can construct a cycle $B_2$ in $\TENS{\theta_2}{\theta}$ by following along a lifting of $P_+$ and then traversing backward along a lifting of $P'$.
    This implies $c(B_2) = c(P_+)-c(W) \geq 0$ and thus $c(W) = c(P_+)$. 

    By suitably following along $C$, i.e., walking along $B$ a suitable number of times, then following along a lifting of $P'$ towards a lifting of $a_{\cut}$ (denote the obtained path by $P''$) and then going backward along a lifting of $P^+$ towards $\hat{s}$, we also obtain a cycle with positive cost in $\TENS{1}{\theta_2^i}_{f_{j_3}}$, i.e. with $c(\pi(P'')) - c(P_+)\geq 0$.
    We can obtain $W$ from $P''$ by traversing along $B$ more often and then following a lifting of $P'$ back to $\hat{s}$ from $\tail(a_+)$.
    Since $c(W) < 0$, this yields $c(P'')-c(P') < 0$.
    Overall, we thus get $c(P')-c(P_+) > 0$, a contradiction.
    Thus, we have $a_{\cut} \neq a_{\cut} '$.
    By assumption there is also a path $P_-$ that ``blocks'' $a_-$ in $\dynN_{\phi(f_{j_3})}$
    Denote by $P_{\rep}$ a path along which $\phi(f_{j_3})$ sends flow through $a'_{\cut}$.
    See Figure~\ref{fig:basecase2} for an illustration of this situation.

    Consider the cycle $B'$ induced by following  $P''$ and then going backward along $P_{\rep}$.
    We again get $c(B') = c(\pi(P'')) - c(P_{\rep}) \geq 0$.
    Also note that $P'$ starts to send flow along $a_{\cut'}$ later than $P_{\rep}$.
    This implies $c(P') \leq c(P_{rep})$
    Thus,
    $c(C) = c(\pi(P'')) - c(P') \geq 0$, a contradiction.

\end{proof}

\newpage
\printbibliography
\newpage
\appendix
\section{Compression procedure}\label{SecA:Compression}
The compression procedure we describe in the following is used in line 28 of our construction.
Let $i+1$ be an iteration of the construction in which Step 2 is visited and let $C$ be a negative cost cycle in $\TEN{\theta}_{f_i}$.

The compression procedure only has an effect on repeated components of $C$ with positive transit time that do not run from layer $\theta_1^i$ to $\theta_2^i$.
We start by describing the compression procedure applied to a repeated component $C'$ of $C$ that ends in layer $\theta_2^i$ and starts in a layer $\theta'$ with $\theta' < \theta_2^i$.
See Figure~\ref{Fig:Compression1} for a visualization.
\begin{figure}[h!]
    \centering
    \includegraphics{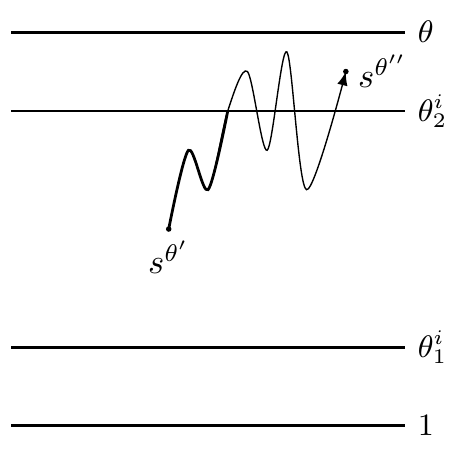}
    \caption{Example for a cycle $C$ in $(\TENS{\theta_1^i}{\theta})^{u=1}$. The compression procedure only has an effect of $C'$ which is indicated by the thick part of $C$ in the figure. }
    \label{Fig:Compression1}
\end{figure}

Define $W := \pi(C')$.
Note that $W$ is a walk in $\dynN_{\phi(f_i)}$ which decomposes into a path $P$ and cycles $B_1,\ldots,B_{\ell}$. 
If $C'$ starts in $\hat{s}$, then $W$ lies completely in $\dynN_{\phi(f_i)}(s)$ by Lemma~\ref{lem:no_cycle_crossing_cuts_pos} while $W$ lies in $\dynN_{\phi(f_i)}(t)$ if $C'$ starts in $\hat{t}$.
In either case, we get that the cycles $B_1,\ldots,B_{\ell}$ are $s$- or $t$-reachable by Lemma~\ref{lem:cycles_are_reachable} and thus, we have $c(B_j) \geq 0$ for all $j \in \{1,\ldots,\ell\}$ by construction, and $c(B_j)>0$ if $\tau(B_j)<0$ for $j \in \{1,\ldots,\ell\}$ by Lemma~\ref{lem:properties}.
In particular, we have $c(P) \leq c(W)$.
The aim of the compression procedure is to construct a new walk $W'$ from $W$ with $c(W') \leq c(W)$ such that we can bound the height of $W'$.
Let $P = (v_1,a_1,v_2,\ldots, a_{k-1},v_k)$.
If $h_{\phi(f_i)}(v_k,P) \geq h_{\phi(f_i)}(v_j,P) $ for all $j \in \{1,\ldots,k-1\}$, then we can lift $P$ back to a path $W'$ in $\TENS{\theta_1^i}{\theta_2^i}_{f_{i}}$ that starts in a layer $\theta' \leq \theta_2^i$ and ends in layer $\theta_2^i$.
Since in this case $W'$ is a path in $\TENS{\theta_1^i}{\theta_2^i}_{f_{i}}$, we get that $h_{f_{i}}^\theta(W') < \sum_{a \in A} \tau(a)$.
See Figure~\ref{Fig:Compression3} for a visualization of this case.
\begin{figure}[h!]
    \centering
    \includegraphics{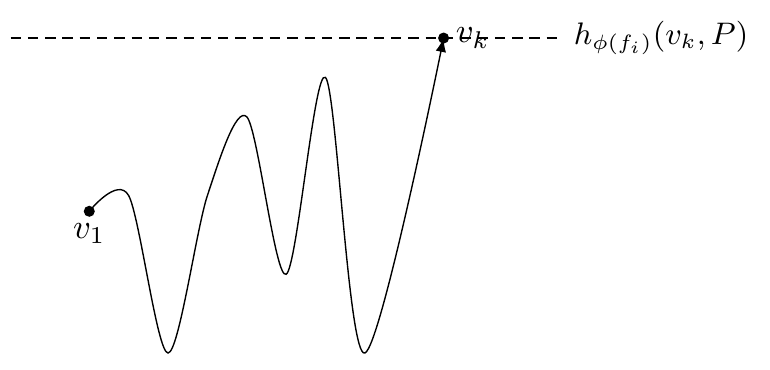}
    \caption{An illustration of a path $P = (v_1,a_1,v_2,\ldots, a_{k-1},v_k)$, where the node $v_k$ has the highest height of all the nodes in $P$.}
    \label{Fig:Compression3}
\end{figure}
In this case the compression procedure returns a lifting of $P$ as a suitable compression of $W$.
Next suppose that there is some $j \in \{1,\ldots,k-1\}$ such that $h_{\phi(f_i)}(v_k,P) < h_{\phi(f_i)}(v_j,P) $. 
See Figure~\ref{Fig:Compression2} for a visualization of this case.
\begin{figure}[h!]
    \centering
    \includegraphics{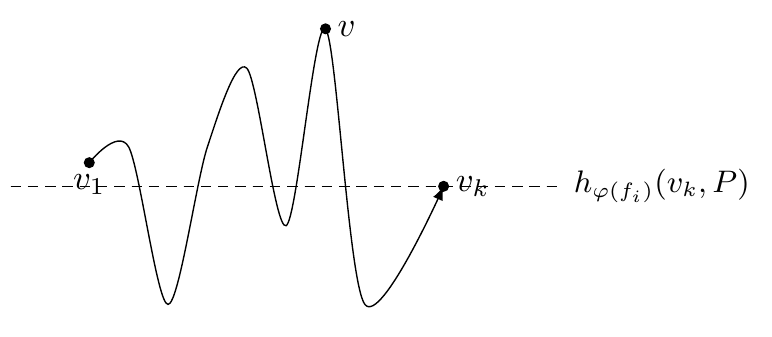}
    \caption{An illustration of a path $P = (v_1,a_1,v_2,\ldots, a_{k-1},v_k)$, where the node $v_k$ does not have the highest height of all the nodes in $P$. }
    \label{Fig:Compression2}
\end{figure}

Then we cannot just lift $P$ back to the time-expanded network in a suitable way.
In $P$ there might be multiple tops of hills that have a higher height than $v_k$.
In Figure~\ref{Fig:Compression2} we have three such hills.
In $W$ there are cycles with negative or positive transit time occurring before $v_k$ or at $v_k$ reducing or increasing the height of every node in $P$ occurring after these cycles.
In particular the height of $v_k$ is affected by such cycles. 
Thus, in order to achieve that the endpoint of $W$ has the highest height among all nodes in $W$, $W$ needs to traverse at least one cycle of positive transit time before or at the node $v_k$ and after the last hill of $P$, i.e., between the node $v$ and $v_k$ in Figure~\ref{Fig:Compression2}.
To obtain the walk $W'$, such that a lifting of $W'$ is returned by our compression procedure  we now enter such cycles into $P$ in such a that $v_k$ has the lowest height of all the points in $W'$ and such that $c(W') \leq c(W)$. 
This is possible by the arguments above.
Since the height of $P$ is at most $\sum_{a \in A} \tau(a)$, the height of $W$ is at most $2 \sum_{a \in A} \tau(a)$.

Next we describe our compression procedure when $C'$ is a repeated component of $C$ that 
starts in layer $\theta_1^i$ and ends in $\hat{s}$ or $\hat{t}$ in layer $\theta'$ with $\theta' > \theta_1^i$.
Define $W := \pi(C')$.
Note that $W$ is a walk in $\dynN_{\phi(f_i)}$ which decomposes into a path $P$ and cycles $B_1,\ldots,B_{\ell}$. 
If $C'$ ends in $\hat{s}$, then $W$ lies completely in $\dynN_{\phi(f_i)}(s)$ by Lemma~\ref{lem:no_cycle_crossing_cuts_pos} while $W$ lies in $\dynN_{\phi(f_i)}(t)$ if $C'$ ends in $\hat{t}$.
In either case, we get that the cycles $B_1,\ldots,B_{\ell}$ are $s$- or $t$-reachable by Lemma~\ref{lem:cycles_are_reachable} and thus, we have $c(B_j) \geq 0$ for all $j \in \{1,\ldots,\ell\}$ by construction, and $c(B_j)>0$ if $\tau(B_j)<0$ for $j \in \{1,\ldots,\ell\}$ by Lemma~\ref{lem:properties}.
In particular, we have $c(P) \leq c(W)$.

The aim of the compression procedure is to construct a new walk $W'$ from $W$ with $c(W') \leq c(W)$ such that we can bound the height of $W'$.
Let $P = (v_1,a_1,v_2,\ldots, a_{k-1},v_k)$.
If $h_{\phi(f_i)}(v_1,P) \leq h_{\phi(f_i)}(v_j,P) $ for all $j \in \{2,\ldots,k\}$, then we can lift $P$ back to a path $W'$ in $\TENS{\theta_1^i}{\theta_2^i}_{f_{i}}$ that starts in a layer $\theta_1^i$  and ends in layer $\theta' \leq \theta_2^i$.
Since in this case $W'$ is a path in $\TENS{\theta_1^i}{\theta_2^i}_{f_{i}}$, we get that $h_{f_{i}}^\theta(W') < \sum_{a \in A} \tau(a)$.

Next suppose that there is some $j \in \{2,\ldots,k\}$ such that $h_{\phi(f_i)}(v_1,P) > h_{\phi(f_i)}(v_j,P) $. 
See Figure~\ref{Fig:Compression4} for a visualization of this case.
\begin{figure}[h!]
    \centering
    \includegraphics{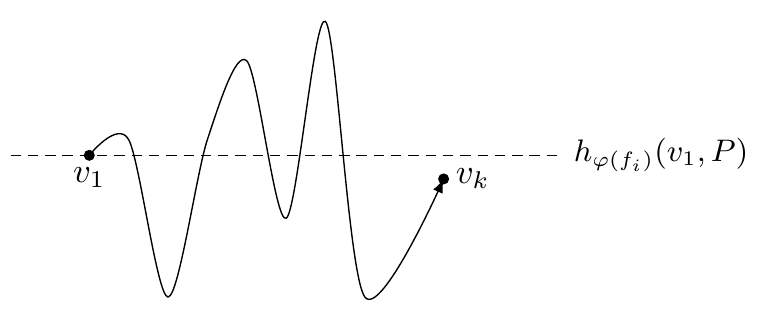}
    \caption{An illustration of a path $P = (v_1,a_1,v_2,\ldots, a_{k-1},v_k)$, where the node $v_1$ does not have the lowest height of all the nodes in $P$. }
    \label{Fig:Compression4}
\end{figure}
In this case we cannot just lift $P$ back to the time-expanded network in a suitable way.
There might be multiple bottoms of valleys in $P$ which have a lower height in $P$ then $v_1$.
In order to ``pull up'' these valleys in $W$, $W$ needs to traverse cycles $B$ of positive transit time. 
Recall that such cycles have cost $c(B)\leq 0$.
Denote by $W'$ the walk obtained by entering cycles with positive transit time from the decomposition of $W$ into $P$ such that in $W'$ the node $v_1$ has the lowest height among all nodes in $W'$ in a minimum cost way.
Note that the height of $W'$ is at most $\mu \cdot \sum_{a \in A} \tau(a)$ where $\mu$ is the number of valleys of $P$.
Since $P$ is a path, there can be at most $|V|$ such valleys.
The walk $W$ might additionally contain cycles of strictly negative transit time which have strictly positive cost. 
Thus, we have $c(W') \leq c(W)$.
Overall, the following lemma summarizes the results of our compression procedure.
\begin{lemma}\label{lem:compression}
    Let $C$ be a negative cost cycle in $\TEN{\theta}_{f_i}$.
    The compression procedure only has an effect on repeated components of $C$ with positive transit time that do not run from layer $\theta_1^i$ to $\theta_2^i$.
    If $C'$ is a repeated component of $C$ that ends in layer $\theta_2^i$ and starts in a layer $\theta'$ with $\theta' < \theta_2^i$, the compression procedure returns a lifting of $W'$ with $h_{\phi(f_i)}(W') \leq 2\sum_{a \in A}\tau(A)$.

   If $C'$ is a repeated component of $C$ that  
starts in layer $\theta_1^i$ and ends in $\hat{s}$ or $\hat{t}$ in layer $\theta'$ with $\theta' > \theta_1^i$, the compression procedure returns a lifting of $W'$ with $h_{\phi(f_i)}(W') \leq |V|\sum_{a \in A}\tau(A)$.
\end{lemma}
\end{document}